\documentclass[a4paper,11pt]{article}

\usepackage{dsfont}
\usepackage[pdftex]{hyperref}
\usepackage{subcaption}
\usepackage{stmaryrd}
\usepackage[normalem]{ulem}
\usepackage{graphicx,xcolor}
\usepackage{amsfonts}
\usepackage{amsmath,amssymb,mathrsfs,amsthm}
\usepackage[all]{xy} 
\usepackage{color}
\usepackage{times}
\usepackage{enumerate,vmargin}
\usepackage{verbatim}
\usepackage{paralist}
\usepackage[labelfont=bf,size=small,font=it]{caption}
\usepackage[numbers,sort&compress]{natbib}
\usepackage{tikz}

\setpapersize{A4}
\setmarginsrb{2.5cm}{2cm}{2.5cm}{2cm}{.4cm}{4mm}{0cm}{7.5mm}

\hypersetup{
    unicode      = false,     
    pdftoolbar   = true,      
    pdfmenubar   = true,      
    pdffitwindow = true,      
    pdfnewwindow = true,      
    colorlinks   = true,      
    linkcolor    = blue,      
    citecolor    = blue,      
    filecolor    = blue,      
    urlcolor     = blue       
}

\theoremstyle{plain}
\newtheorem{The}{Theorem}
\newtheorem{Def}{Definition}

\newtheorem{Prop}[The]{Proposition}
\newtheorem{Lem}[The]{Lemma}

\newtheorem{Alg}{Algorithm}

\theoremstyle{remark}
\newtheorem{Req}{Remark}

\newcommand{\R}{\mathbb R}
\newcommand{\N}{\mathbb N}
\newcommand{\Z}{\mathbb Z}

\newcommand{\cO}{\mathcal O}
\newcommand{\cI}{\mathcal I}
\newcommand{\cJ}{\mathcal J}

\newcommand{\cN}{\mathcal N}
\newcommand{\cB}{\operatorname{Bin}}
\newcommand{\cG}{\mathcal G}
\newcommand{\cK}{\mathcal K}
\newcommand{\cF}{\mathcal F}
\newcommand{\cD}{\mathcal D}
\newcommand{\cH}{\mathcal H}

\newcommand{\cA}{\mathcal A}
\newcommand{\cX}{\mathcal X}

\newcommand{\Var}{\operatorname{Var}}

\newcommand{\vep}{\varepsilon}

\definecolor{green}{RGB}{0, 128, 128}
\definecolor{darkred}{RGB}{139,0,0}

\title{\textsc{Colour ratio in Prim's ranking of bipartite graphs}
\thanks{Supported by the EPSRC EP/W033585/1 grant}
}

\author{F\'elix \textsc{Kahane}
\thanks{Department of Mathematics, University of Sussex, 
Falmer campus, Brighton, BN1 9QH, England, United Kingdom.   
Email: F.Kahane@sussex.ac.uk} \ and\  
Minmin \textsc{Wang}
\thanks{Department of Mathematics, University of Sussex, 
Falmer campus, Brighton, BN1 9QH, England, United Kingdom.  
Email: Minmin.Wang@sussex.ac.uk}}

\date{\today}

\begin{document}

\maketitle
\begin{abstract}
We consider a complete bipartite graph of size $n$ endowed with i.i.d.~uniform edge weights and run Prim's Algorithm to obtain a ranking of its vertices. Let $\rho^{(n)}_k$ be the proportion of black vertices among the first $k$ vertices in this ranking. We characterise the limit behaviour of $\rho^{(n)}_k$ as both $n$ and $k$ tend to infinity. Our results show that in general the limit of $\rho^{(n)}_k$, when existing, differs from the overall proportion of the black vertices in the graph. 


\smallskip

\end{abstract}


\section{Introduction and main results}

This work concerns minimum spanning tree. 
Let $G=(V, E)$ be a connected graph and suppose that each edge $e\in E$ is associated with an {\it edge weight} $u_e\in (0, \infty)$. We will call $(G, (u_e)_{e \in E})$ a ~{\it weighted graph}. For a subgraph $G'$ of $G$, the {\it weight of $G'$} is the sum of the edge weights over all the edges contained in $G'$.
A {\it minimum spanning tree} (MST) of the weighted graph $(G, (u_e)_{e\in E})$ is a subgraph $T$ of $G$ that has the minimal weight among all the connected subgraphs of $G$ on the vertex set $V$. It is not difficult to see that $T$ is necessarily a tree, i.e.~cycle-free. If the edge weights are all distinct, then there is a unique minimum spanning tree.

When $G$ is $K_n$, the complete graph of $n$ vertices,  and the edge weights are i.i.d.~random variables, the corresponding minimum spanning tree has been extensively studied in the literature; see~\cite{Frieze1985, Janson1995, addario-berry_scaling_2017} to name only a few. 
Related models are also abundant, including MST of regular graphs~\cite{BeFrMc1998}, of the hypercubes~\cite{Penrose1998}, lattices~\cite{GaPeSc2018}, as well as the Euclidean MST~\cite{BeHaHa1959} among others.   

We are interested in the minimum spanning tree of a complete bipartite graph. 
We write $|A|$ for the cardinality of a finite set $A$. Let $n_w, n_b$ be two natural numbers and $n= n_b+n_w$. From now on, we take $G$ to be the complete bipartite graph $K_{n_b,n_w}$ on the vertex set $V_n=V^b_{n}\cup V^w_{n}$ with $|V^b_{n}|=n_b$ and $|V^w_{n}|=n_w$. 
In the sequel, we will refer to vertices from $V^b_{n}$ (resp.~vertices from $V^w_{n}$) as {\it black vertices} (resp.~{\it white vertices}).
We will be interested in the case where the size of the graph tends to infinity with a fixed colour ratio. More specifically, we assume that 
\begin{equation}
\label{hyp: theta}
\exists\,\theta\in (0, 1): \quad \frac{n_b}{n}\xrightarrow{n\to\infty} \theta.
\end{equation}

Denote by $E_{n}=\{e=\{v, w\}, v\in V^b_n, w\in V^w_n\}$ the edge set of $K_{n_b, n_w}$ and let $\{U_e: e\in E_{n}\}$ be a collection of i.i.d~random variables with a common uniform distribution on $(0, 1)$. We observe that with probability 1, $U_e$'s are all distinct. The (a.s.~unique) minimum spanning tree of $(K_{n_b, n_w}, \{U_e: e\in E_{n}\})$ can be identified using Prim's Algorithm as follows. 

\begin{Alg}[Prim's Algorithm]
\label{algo: Prim}
    Given $K_{n_b, n_w}=(V_n, E_n)$ and the edge weights $\{U_e: e\in E_{n}\}$, define a sequence of vertices $(\sigma(k))_{1\le k\le n}$ and a sequence of edges $(e_k)_{1\le k\le n-1}$ as follows.

    \noindent
    {\bf Step 1. } Let $\sigma(1)$ be a uniform element of $V_n$.

    \noindent
    {\bf Step $\boldsymbol{k\ge 2}$. } Given $\Sigma(k-1):=\{\sigma(i): 1\le i\le k-1\}$ and $\{e_{i}: 1\le i\le k-2\}$, let $e_{k-1}$ be the a.s.~unique edge that satisfies
    \[
        U_{e_{k-1}} = \min\big\{U_e: e=\{v, v'\}, v\in \Sigma(k-1), v'\in V_n\setminus \Sigma(k-1)\big\}.
    \]
    Set $\sigma(k)$ to be the endpoint of $e_{k-1}$ that does not belong to $\Sigma(k-1)$.

    \noindent
    {\bf Stop at step $\boldsymbol{n}$} and return the graph $T_n$ on the edge set $\{e_k: 1\le k\le n-1\}$.
\end{Alg}

For $1\le k\le n$, denote by $T_k$ the graph with the vertex set $\Sigma(k)$ and edge set $\{e_i: 1\le i\le k-1\}$. It is not difficult to see that $T_k$ is a connected graph with $k$ vertices and $k-1$ edges, i.e.~a tree. Algorithm~\ref{algo: Prim} can be described as creating a subgraph $T_{k+1}$ from $T_k$ by adding a new vertex with the minimum increase of weights, for $1\le k<n$. 

It is known that $T_n$ is the minimum spanning tree of $(K_{n_b, n_w}, \{U_e: e\in E_n\})$ (\cite{Prim}). Furthermore, Algorithm~\ref{algo: Prim} also produces an ordered sequence $(\sigma(k))_{k\in [n]}$ from the $n$ vertices of $K_{n_b, n_w}$, which we will refer to as the {\it Prim sequence} of the vertices of $K_{n_b, n_w}$, as well as a sequence of $n-1$ edges $(e_k)_{k\in[n-1]}$, called {\em Prim edges} in the sequel. In the same terminology, $k$ will be referred to as {\em Prim's rank} or simply the {\em rank} of $\sigma(k)$. 

For $k\in [n]$, we denote 
\begin{equation}
\label{def: Sigma-set}
\Sigma^b(k) = \Sigma(k) \cap V^b_n \quad\text{and}\quad\Sigma^w(k) = \Sigma(k) \cap V^w_n,
\end{equation}
the respective subsets of black and white vertices of $T_k$.  Our main object of interest is the following random variable that represents the proportion of black vertices in $T_k$:

\begin{equation}
    \label{def: ratio}
    \rho^{(n)}_k := \frac{|\Sigma^b(k)|}{k}. 
\end{equation}
Clearly, $\rho^{(n)}_n$ is simply the overall proportion of black vertices. 
The aim of this work is to investigate the asymptotic behaviors of $\rho^{(n)}_k$ as both $n$ and $k$ grow to infinity, possibly at different speeds. Recall from~\eqref{hyp: theta} that the overall proportion of black vertices converges to $\theta$. 
However, our main results below show that unless $\theta=1/2$, i.e.~asymptotic parity for black vs white vertex numbers, we will observe a different proportion than $\theta$ along the Prim sequence. More specifically, we distinguish two regimes: the {\it sublinear regime} where $k\to \infty$ but $k=o(n)$ as $n\to\infty$, and the {\it linear regime} where $k$ grows linearly with $n$.

\begin{The}[Sublinear regime]

\label{thm: sublinear}
Assume that~\eqref{hyp: theta} holds.
Denote $\gamma_\theta = \sqrt{\frac{1-\theta}{\theta}}$. 
Let $(\kappa_n)_{n\ge 1}$ be a sequence of positive integers that tends to infinity and satisfies $\kappa_n/n\to 0$ as $n\to\infty$. Then
    \begin{equation}
        \label{eq: thm1}
        \rho^{(n)}_{\kappa_n} \xrightarrow{n\to\infty} \frac{1}{1+\gamma_\theta}=\frac{\sqrt\theta}{\sqrt\theta+\sqrt{1-\theta}} \quad\text{in probability}. 
    \end{equation}
\end{The}

\begin{The}[Linear regime]
\label{thm: linear}
Assume that~\eqref{hyp: theta} holds.
Let $\alpha_\theta = (\theta (1-\theta))^{-1/2}$.  For each $\lambda>1$, let $(\ell(\lambda), \rho(\lambda))$ be the unique solution in $(0, 1)\times (0, 1)$ of the following equation:
\begin{equation}
\label{eq: fixpt-thm}
\left\{
\begin{array}{cl}
\rho\ell & = \theta (1 - e^{-\alpha_{\theta}\lambda(1-\rho)\ell}) \\
(1-\rho)\ell & = (1-\theta) (1 - e^{-\alpha_{\theta}\lambda \rho\ell})
\end{array}
\right.
\end{equation}
Then $\lambda \mapsto \ell(\lambda)$ is a bijection from $(1, \infty)$ to $(0, 1)$. Denote by $\ell^{-1}$ the inverse function. For each $s\in (0, 1)$, we have
\begin{equation}
\label{eq: thm2}
\rho^{(n)}_{\lfloor sn\rfloor } \xrightarrow{n\to\infty} \rho\circ\ell^{-1}(s) \quad\text{in probability}.   
\end{equation}
\end{The}

\begin{figure}[ht]
    \centering
    \subfloat[]{%
        \includegraphics[width=6.8cm]{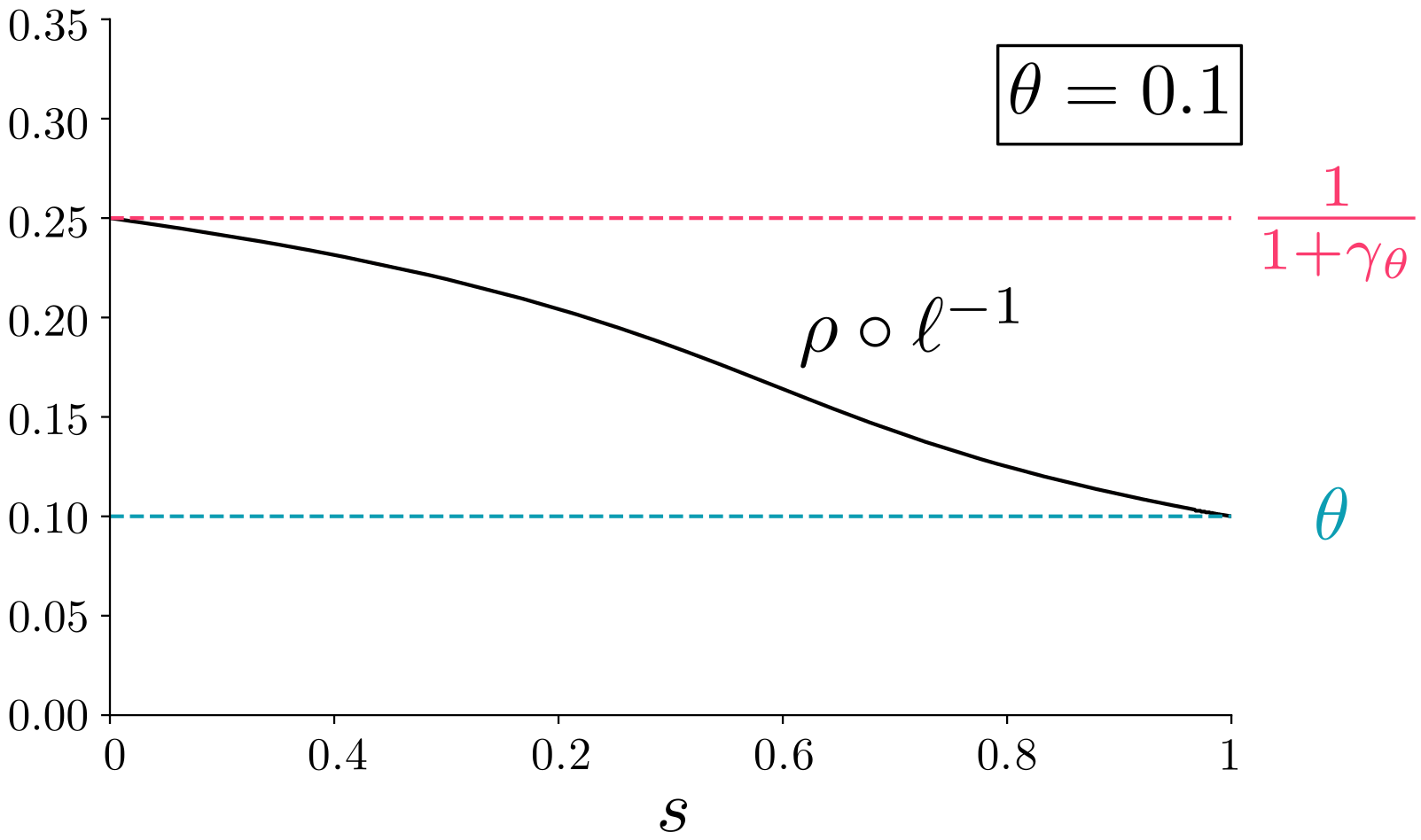}%
    }\qquad
    \subfloat[]{%
        \includegraphics[width=6.8cm]{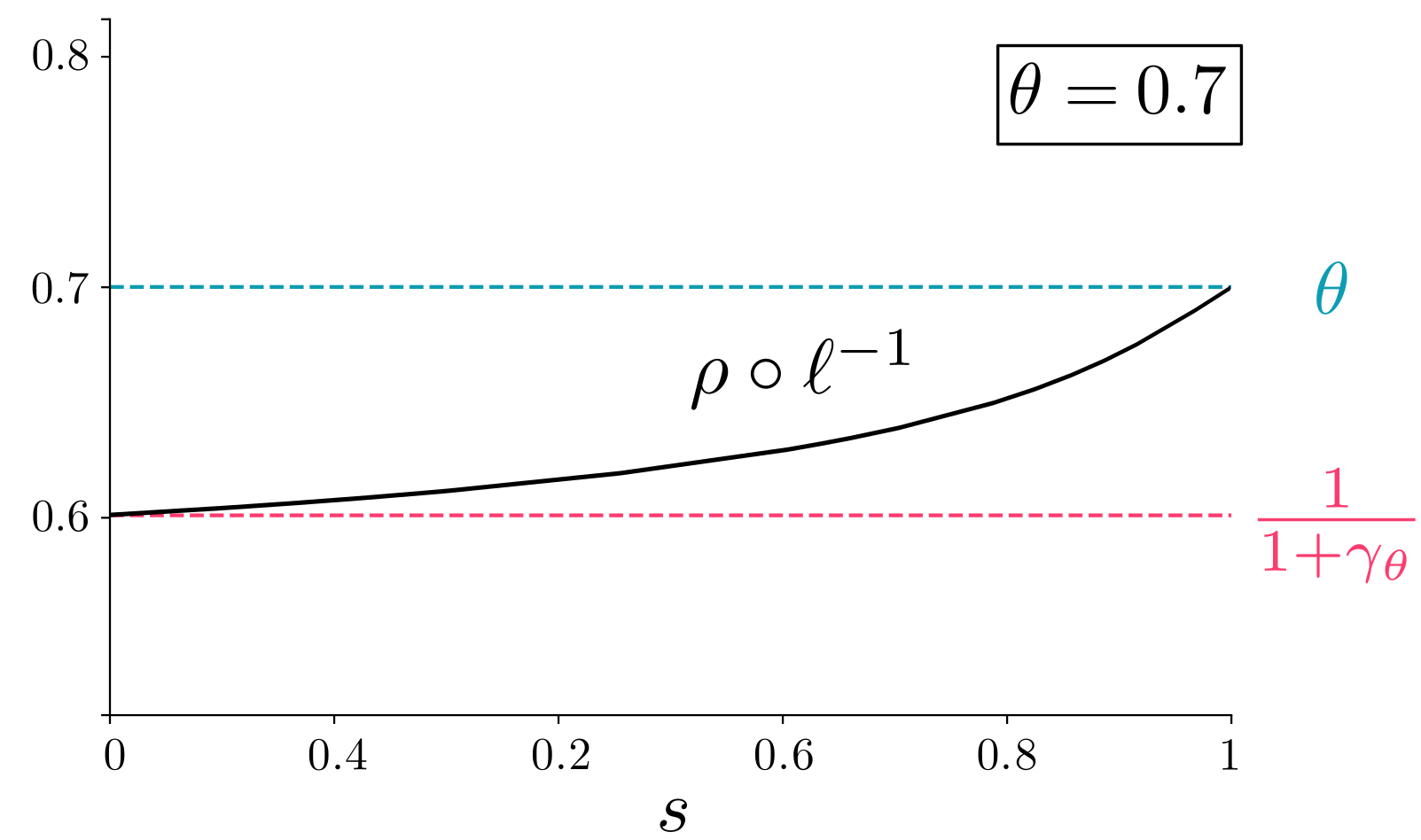}%
    }
    \caption{Plot of the function $s\mapsto \rho\circ \ell^{-1}(s)$ in the cases (a) $\theta=0.1$ and (b) $\theta=0.7$.}
    \label{fig:plot-fun}%
\end{figure}

In the sublinear regime, we note that the limit in~\eqref{eq: thm1} only depends on $\theta$ and satisfies that 
\[
\text{for } \theta\in \Big(0, \frac12\Big): \theta<\frac{1}{1+\gamma_{\theta}}<\frac12, \quad\text{and for } \theta \in \Big(\frac12, 1\Big): \frac12<\frac{1}{1+\gamma_{\theta}}<\theta.
\]
By contrast, in the linear regime, the limit in~\eqref{eq: thm2} depends on both $\theta$ and $s\in (0, 1)$; as shown in Fig.~\ref{fig:plot-fun}, it interpolates between $(1+\gamma_\theta)^{-1}$ and $\theta$ as $s$ increases. We also point out that Lemma~\ref{lem: solv} below shows that  $(1 + \gamma_\theta)^{-1} = \rho \circ \ell^{-1}(0+)$, so that the transition is continuous as we approach the  sublinear regime from the linear regime.

\begin{Req}[Influence of the starting vertex] 
\label{req1}
In Algorithm~\ref{algo: Prim}, we have taken $\sigma(1)$ to be a uniform vertex. Since under Assumption~\eqref{hyp: theta}, the probability of $\sigma(1)$ being a black vertex converges to a non-zero quantity, both limits in~\eqref{eq: thm1} and~\eqref{eq: thm2} still hold if $\sigma(1)$ is chosen to be a uniform black (resp.~white) vertex. Finally, by exchangeability, Theorems~\ref{thm: sublinear} and~\ref{thm: linear} remain valid if $\sigma(1)$ is any fixed vertex. 
\end{Req}

\begin{Req}[Interpretation in invasion percolation]
Prim's Algorithm is equivalent to the model of invasion percolation (\cite{WiWi1983, Stark1991}). The model is well-defined for a general weighted graph, but let us take the case of $(K_{n_b, n_w}, (U_e)_{e\in E_n})$ to illustrate. In the context of invasion percolation on $K_{n_b,n_w}$, the focus is on the sequence of subgraphs $(T_k)_{k\in [n]}$, produced by Algorithm~\ref{algo: Prim}. The evolution of $(T_k)_{k\in [n]}$ is meant to represent a growing cluster inside a disordered medium, always choosing the easiest next location to invade. As such, $(\rho^{(n)}_k)_{k\in [n]}$ tracks the proportion of black vertices in the cluster during the percolation. 
\end{Req}

Our approach to the problem is based upon a connection between Prim's Algorithm and the Bernoulli bond percolation of $K_{n_b, n_w}$. In fact, as we shall see, the pair $(\ell(\lambda), \rho(\lambda))$ in Theorem~\ref{thm: linear} is linked to the size of the giant component in the bond percolation. Let us briefly describe this connection here. 

Let $p\in (0, 1)$ and let $G(n_b, n_w, p)$  be the subgraph of $K_{n_b, n_w}$ with the vertex set $V_n$ and the edge set
\begin{equation}
\label{def: Enp}
E_n(p) =\{e\in E_n: U_e\le p\}. 
\end{equation}
Since the edge weights are i.i.d.~uniformly distributed, each edge of $K_{n_b, n_w}$ is present in $G(n_b, n_w, p)$ independently with probability $p$. 
The graph $G(n_b, n_w, p)$ can have several connected components, and the vertex sets of these components have a simple representation in the Prim sequence $(\sigma(i))_{i\in [n]}$. Indeed, recall the Prim edges $(e_i)_{i\in [n-1]}$ and let 
\[
J_1 = \inf\big\{ i\in [n-1]: U_{e_i} >p\} \wedge n;
\]
then the vertex set of the component of $G(n_b, n_w, p)$ containing $\sigma(1)$ is given by $\{\sigma(i): 1\le i\le J_1\}$. More generally, for $j\ge 2$, define
\[
J_j=\inf\{ i > J_{j-1}:  U_{e_i}>p\} \wedge n \quad \text{and}\quad m = \inf\{ j\ge 1: J_j=n\}.
\]
Proposition~\ref{prop: prim} below says that there are $m$ connected components  in $G(n_b, n_w, p)$, and the vertex set of the $j$th component is given by 
\begin{equation}
\label{id: cc-int}
\mathcal C_j(p):=\{\sigma(i): J_{j-1}+1\le i\le J_j\}, \quad 1\le j\le m,
\end{equation}
with the convention that $J_0=0$. This shows that the connected components of $G(n_b, n_w, p)$ always consist of vertices of {\it consecutive} Prim's ranks. 

The above representation of the connected components as intervals in Prim's ranking is known for various graphs and has led to many interesting applications; see for instance~\cite{Al1995, LyPeSc2006, DaSaVa2009, addario-berry_scaling_2017, GaPeSc2018, ABBa2023}. 
Its current form for the complete bipartite graph underpins our proof of
Theorem~\ref{thm: linear}, where we tune $p$ to a value in the supercritical regime that will match the size of the giant component to $sn$. We then rely on results about the giant component in $G(n_b, n_w, p)$ to prove the convergence in~\eqref{eq: thm2}. 

For the proof of Theorem~\ref{thm: sublinear}, we require a stronger coupling between the Prim sequence and the bond percolation. More specifically, we adapt a graph exploration process from Broutin and Marckert~\cite{broutin_new_2016} to the bipartite case, and apply it to $G(n_b, n_w, p)$ with $p$ at criticality. The convergence in~\eqref{eq: thm1} is then proved after a careful analysis of the stochastic processes that encode the graph exploration.

\paragraph{Discussion.} We discuss here some follow-up questions to Theorems~\ref{thm: sublinear} and~\ref{thm: linear}. 
\begin{itemize}
\item 
{\bf Fluctuations and Central limit theorems.} Both our main theorems can be viewed as a law-of-large-numbers type result. It is then natural to query the convergence rates in~\eqref{eq: thm1} and~\eqref{eq: thm2}, under potentially stronger assumptions. For the sublinear regime, this is currently under investigation by the first author in~\cite{}. For the linear regime, we believe that the fluctuations are of order $\sqrt{n}$, and may be deduced from the results in~\cite{Clancy25+} and a refinement of our arguments in Section~\ref{sec: linear}. 

\item 
{\bf General colour schemes and branching process connection.} As the graph $G(n_b, n_w, p)$ can be approximated with a two-type branching process (see Section~\ref{sec: percolation} for more detail), Theorem~\ref{thm: linear} is related to the well-known convergence of types theorem for {\it supercritical} multi-type branching processes (\cite{AN_book}, V.6, Theorem 1), while Theorem~\ref{thm: sublinear} is in a similar spirit to Theorem 1 of~\cite{Mi08} for {\it critical} branching processes. This leads us to believe that our proof can be readily adapted to a $d$-partite graph for any $d\ge 2$, and more generally  the stochastic block model as studied in~\cite{Clancy25+}.  

\item 
{\bf Uniform convergence and exceptional vertices.} We use the notation from Theorems~\ref{thm: sublinear} and~\ref{thm: linear}. As pointed out in Remark~\ref{req1}, the convergences in~\eqref{eq: thm1} and~\eqref{eq: thm2} also hold if we let $\sigma(1)$ be any fixed vertex. For $v\in V_n$, denote by $(\sigma_{v}(i))_{i\in [n]}$ the Prim sequence starting from $\sigma(1)=v$, and by 
$\rho^{(n)}_k(v)$ the proportion of black vertices among $\{\sigma_{v}(i):i\in [k]\}$. The following question has been put to us by Vlad Vysotskiy: is it true that
\[
\sup_{v\in V_n} \big|\rho^{(n)}_{sn}(v)-\rho\circ \ell^{-1}(s)\big| \xrightarrow{n\to\infty}0 \  \text{ and } \sup_{v\in V_n} \Big|\rho^{(n)}_{\kappa_n}(v)-\frac{1}{1+\gamma_{\theta}}\Big|\xrightarrow{n\to\infty} 0 \text{ in probability}?
\]
We believe that the first convergence will hold but the second will fail for $(\kappa_n)_{n\in \N}$ slowly enough. However, proving the claim is far beyond the scope of the current work; we thus leave it as open questions. 
\end{itemize}

\paragraph{Organisation of the paper.}
In Section~\ref{sec: explore}, we introduce the specific graph exploration that is key to our approach. After presentation of the various results on the bond percolation in Section \ref{sec: percolation}, we prove Theorem~\ref{thm: linear} in Section~\ref{sec: linear}. 
In Section~\ref{sec: neighbour}, we study the distributional properties of the graph exploration processes, paving the way for our proof of Theorem~\ref{thm: sublinear} in Section~\ref{sec: sublinear}. Appendix~\ref{sec: app} collects some technical results used in the proof. 

\paragraph{Notation.}
For $n\in \N$, we denote $[n]=\{1, 2, 3, \dots, n\}$ and use the convention that $[0]=\varnothing$. 
Unless said otherwise, all the random variables that appear so far and below are defined on the same probability space $(\Omega, \mathscr F, \mathbb P)$. 
Let $(X_n)_{n\in \N}$ be a sequence of real-valued random variables and $(a_n)_{n\in \N}$ a sequence of positive real numbers.  We denote 
\[
X_n = o_p(a_n), \; n\to \infty \quad\text{if}\quad \frac{X_n}{a_n} \xrightarrow{n\to\infty} 0 \quad\text{in probability.}
\]
Similarly, we write 
\[
X_n = O_p(a_n), \; n\to \infty \quad\text{if}\quad \bigg\{\frac{X_n}{a_n}: n\ge 1\bigg\} \text{ is tight in }\R. 
\]

\section{Graph exploration in Prim's ranking}
\label{sec: explore}

The aim of this section is two-fold. Firstly, we introduce a generalised graph exploration and discuss the properties of the vertex ordering induced by the exploration. This is largely inspired by Broutin and Marckert~\cite{broutin_new_2016}, although the presentation here is new. Secondly, we explain, in the specific case of $K_{n_b, n_w}$, how this leads to the representation~\eqref{id: cc-int} of the connected components in $G(n_b, n_w, p)$, which is the starting point for our proof of Theorem~\ref{thm: linear}, as well as a 2-neighbourhood exploration of $G(n_b, n_w, p)$ that is the main tool in our proof of Theorem~\ref{thm: sublinear}. 

\subsection{Vertex ordering and graph exploration order}
\label{sec: prim}

Let $V$ be a nonempty finite set. An {\it ordering} on $V$ is a bijection $\pi: \{1, 2, 3, \dots, |V|\}\to V$. For each $v\in V$, $\pi^{-1}(v)\in \N$ is called the {\it $\pi$-rank} of $v$. If $1\le i\le j\le |V|$, we define the following subset of $V$:
\[
\llbracket i, j\rrbracket_{\pi} = \big\{\pi(k): i\le k\le j\big\},
\]
which consists of elements ranked between $i$ and $j$. 

Here we are interested in a specific class of ordering. 
To introduce the class, we need the following notion. For a graph $G=(V, E)$ and a subset of vertices $W\subseteq V$, we say that $W$ is {\it connected} if the induced subgraph of $G$ on $W$ is connected, and two subsets $W, W'$ are {\it disconnected} from each other if there is no edge of the form $\{w, w'\}$ with $w\in W$, $w'\in W'$.

\begin{Def}[Graph exploration order (GEO)]
Let $G=(V, E)$ be a (finite) graph. An ordering $\pi$ on $V$ is said to be a  \emph{graph exploration order}, GEO in short, for $G$ if it satisfies the following property: for all $1\le i\le j\le |V|$, if $\pi(i)$ and $\pi(j)$ belong to the same connected component of $G$, then there exists some $i'\le i$ so that $\llbracket i', j\rrbracket_{\pi}$ is connected.
\end{Def}

If $\pi$ is the ordering induced by the usual breadth-first or depth-first exploration of a graph, then it is a GEO. Indeed, in these explorations, we extract a spanning tree from each connected component and the ancestors of each vertex in this tree are always explored before the vertex. Therefore, if $\pi(i)$ and $\pi(j)$ belong to the same connected component, then by taking $i'$ to be the $\pi$-rank of the root of the spanning tree of the component, we see that $\llbracket i', j\rrbracket_{\pi}$ contains all the ancestors of $\pi(i)$ and $\pi(j)$. 

\medskip

For a graph $G=(V, E)$ and two vertices $u, v\in V$, a {\it path from $u$ to $v$} is a finite sequence of vertices $(w_j)_{1\le j\le k}$ satisfying $w_1=u, w_k=v$ and $\{w_j, w_{j+1}\}\in E$ for all $1\le j\le k-1$. 
Let $\pi$ be an ordering of $V$. Suppose that $u, v$ are two vertices of $G$ satisfying $\pi^{-1}(u)<\pi^{-1}(v)$. 
We say that a path $(w_j)_{1\le j\le k}$ from $u$ to $v$ is {\it $\pi$-good} if 
\[
\pi^{-1}(w_1)<\pi^{-1}(w_2)<\cdots<\pi^{-1}(w_k).
\]
The following lemma summarizes the features of a GEO that are used in the main proofs.
\begin{Lem}
\label{lem: prim}
Let $G=(V, E)$ be a graph and $\pi$ a GEO for $G$. Let $C=(V_C, E_C)$ be a connected component of $G$. The following statements hold true.
\begin{enumerate}[(i)]
\item
Let $k_-^C=\min\{\pi^{-1}(v): v\in V_C\}$ and  $k_+^C=\max\{\pi^{-1}(v): v\in V_C\}$; then we have 
\[
V_C=\big\llbracket k_-^C,  k_+^C\big\rrbracket_{\pi}.
\]
\item 
For every $v\in V_C$, there exists a $\pi$-good path from $\pi(k_-^C)$ to $v$. 
\end{enumerate}
Moreover, if $\pi$ is an ordering of $V$ such that (i) and (ii) hold for each connected component of $G$, then $\pi$ is a GEO. 
\end{Lem}

\begin{proof}
Let us first show that if $\pi$ is a GEO, then (i) and (ii) hold. We denote $W=\llbracket k_-^C,  k_+^C\rrbracket_{\pi}$. 
Applying the definition of GEO to $k_-^C$ and $k_+^C$, we can find some $k\le k_-^C$ so that $\llbracket k, k_+^C\rrbracket_{\pi}$ is connected. However, $k_-^C$ being the lowest $\pi$-rank in $V_C$ implies that $k=k_-^C$, and therefore $W\subseteq V_C$. On the other hand, if $v\in V_C\setminus W$, then either $\pi^{-1}(v)<k_-^C$ or $\pi^{-1}(v)>k_+^C$, contradicting in either case the definition of $k_-^C$ or $k_+^C$. Hence, we must have $W=V_C$, which proves {\it (i)}. For {\it (ii)}, we will argue by an induction on $k=\pi^{-1}(v)-k_-^C$, the relative rank of $v$ in $V_C$. If $k=1$, i.e.~$v=\pi(k_-^C)$, this is clearly true. 
Suppose that the statement holds for any vertex of rank up to $k_-^C+k$ and $|V_C|\ge k+k_-^C+1$. Let us consider $v=\pi(k_-^C+k+1)$. Applying the definition of GEO to $k_-^C$ and $k_-^C+k+1$, we can find some $i\le k_-^C$ so that $\llbracket i, k_-^C+k+1\rrbracket_{\pi}$ is connected. However, $k_-^C$ being the lowest rank among the vertices of $V_C$ means that it must be $i=k_-^C$. Hence, we can find some $u\in \llbracket k_-^C, k_-^C+k\rrbracket_{\pi}$ that is adjacent to $v$. By the induction assumption, there is a $\pi$-good path from $\pi(k_-^C)$ to $u$. Appending $v$ to that path yields a $\pi$-good path from $\pi(k_-^C)$ to $v$. This proves {\it (ii)}. 

Now suppose that {\it (i)} and {\it (ii)} hold for an ordering $\pi$. Suppose that $1\le i\le  j\le |V|$ and $\pi(i), \pi(j)$ both belong to the connected component $C=(V_C, E_C)$. Then {\it (i)} implies $\llbracket k_-^C, j\rrbracket_{\pi}\subseteq V_C$. Now applying {\it (ii)} to each $v\in \llbracket k_-^C, j\rrbracket_{\pi}$ shows that the set is connected. We conclude that $\pi$ is a GEO. 
\end{proof}

Let us recall from Algorithm~\ref{algo: Prim} the Prim sequence $(\sigma(i))_{i\in[n]}$ and the Prim edge $e_{k-1}$ which connects $\sigma(k)$ to some $\sigma(k')$ with $k'<k$. Recall also the graph $G(n_b, n_w, p)=(V_n, E_n(p))$ with the edge set $E_n(p)$ defined in~\eqref{def: Enp}. 

\begin{Prop}
\label{prop: prim}
For each $p\in (0, 1)$, the vertex sets of the connected components of $G(n_b, n_w, p)$ are given by $\mathcal C_j(p), 1\le j\le m$, as defined in~\eqref{id: cc-int}, and the Prim sequence $\sigma: [n]\to V_n$ defines a GEO for $G(n_b, n_w, p)$. 
\end{Prop}

\begin{proof}
We first prove the representation of the connected components. To ease the notation, we will simply write $\mathcal C_j$ instead of $\mathcal C_j(p)$, defined as $\llbracket J_{j-1}+1, J_j\rrbracket_{\sigma}$. For $j=1$, as long as $J_1>1$, $e_1\in E_n(p)$ and therefore $\sigma(2)$ is connected to $\sigma(1)$ via $e_1$. Proceeding in this way, we find that for $1\le i\le J_1$, $\sigma(i)$ is connected to $\llbracket 1, i-1\rrbracket_{\sigma}$. It follows that $\mathcal C_1$ is connected in $G(n_b, n_w, p)$. If $m=1$, then there is nothing to show. Assume that $m>1$. Since $e_{J_1}$ has the smallest edge weight among all the edges connecting $\mathcal C_1$ to $V_n\setminus \mathcal C_1$ and $U_{e_{J_1}}>p$, we see that none of these edges is present in $E_n(p)$. Hence, $\mathcal C_1$ is disconnected from $V_n\setminus \mathcal C_1$, which shows that $\mathcal C_1$ is a connected component. 

More generally for $1\le j\le m-1$, suppose that we have shown that $\mathcal C_{j'}, j'\le j$, are connected components. This implies in particular that $\cup_{j'\le j}\mathcal C_{j'}$ are disconnected from $\cup_{j'> j}\mathcal C_{j'}$. On the other hand, if $J_{j+1}-J_{j}>1$, then the edge $e_{J_j+1}$ is present in $E_n(p)$, connecting $\sigma(J_j+2)$ to $\Sigma_{J_j+1}=\cup_{j'<j}\mathcal C_{j'}\cup \{\sigma(J_j+1)\}$. Combined with the previous argument, it can only be that the other endpoint of $e_{J_j+1}$ is $\sigma(J_j+1)$. Proceeding in this way, we find that for $J_j< i\le J_{j+1}$, $\sigma(i)$ is connected to $\llbracket J_j+1, i-1\rrbracket_{\sigma}$. Hence, $\mathcal C_{j+1}$ is connected in $G(n_b, n_w, p)$. The definitions of $e_{J_{j+1}}$ and $J_{j+1}$ then imply that $\mathcal C_{j+1}$ must be disconnected from $\cup_{j'>j+1}\mathcal C_{j'}$. This shows that $\mathcal C_j, 1\le j\le m$, are the vertex sets of the connected components of $G(n_b, n_w, p)$. 

To see why $\sigma$ is a GEO, let $1\le i_1<i_2\le n$ and suppose that $\sigma(i_1), \sigma(i_2)$ belong to the same connected component of $G(n_b, n_w, p)$. Thanks to the representation~\eqref{id: cc-int}, we can find some $j\in [m]$ so that $J_j+1\le i_1<i_2\le J_{j+1}$. Take $i'=J_j+1$; then the previous arguments have shown that $\llbracket J_j+1, i_2\rrbracket_{\sigma}$ is connected. This concludes the proof. 
\end{proof}

Let $G=(V, E)$ be a graph and $\pi$ be an ordering on $V$. The following algorithm is an extension of the habitual breadth-first and depth-first explorations. 
For a subset of vertices $A\subseteq V$, we define its {\em $1$-neighborhood} as follows:   

\begin{equation}
\label{def: N1}
\mathcal N(A)=\big\{v\in V\setminus A: \exists\, a\in A \text{ such that }\{a, v\}\in E\big\}.
\end{equation}
If $A=\{v\}$, we often write $\cN(v):=\mathcal N(\{v\})$. 

\begin{Alg}[Graph exploration in $\pi$-ordering]
\label{algo: graph}
Let $\hat\pi: \{1, 2, \dots, |V|\}\to V$ be defined as follows.
    \noindent
    {\bf Step 1. } Set $\hat\pi(1)=\pi(1)$ and let $\hat{\mathcal{A}}_1 = \mathcal{N}(\hat\pi(1))$.
    
    \noindent
    {\bf Step $\boldsymbol{2\le k\le |V|}$. } 
    If $\hat{\mathcal{A}}_{k-1}$ is nonempty, set $\hat{\pi}(k)$ to be the element of $\hat{\mathcal A}_{k-1}$ with the lowest $\pi$-rank; otherwise, set $\hat{\pi}(k)$ to be the element of $V\setminus\{\hat\pi(1), \hat\pi(2),\dots, \hat\pi(k-1)\}$ with the lowest $\pi$-rank. Let $\hat{\mathcal{A}}_k =\mathcal N(\{\hat\pi(1), \hat\pi(2),\dots,\hat\pi(k)\})$.
\end{Alg}

\begin{Prop}
    \label{luka-preserver-graph-explo}
Let $\hat\pi$ be as in Algorithm~\ref{algo: graph}.   If $\pi$ is a GEO for $G$, then $\hat\pi=\pi$.
\end{Prop}

\begin{proof}
Let us denote $k_{\ast}=\max\{k\in [n]: \pi(k)=\hat\pi(k)\}$. We note that since $\pi(1)=\hat\pi(1)$, $k_{\ast}>1$. Suppose that $k_{\ast}<|V|$. If $\hat{\mathcal{A}}_{k_{\ast}}=\varnothing$, then by definition $\hat\pi(k_{\ast}+1)$ is the lowest $\pi$-ranking element of $V\setminus\{\hat\pi(i): 1\le i\le k_{\ast}\}=\{\pi(i): k_{\ast}<i\le |V|\}$, which is $\pi(k_{\ast}+1)$. This contradicts with the choice of $k_{\ast}$; so we must have $\hat{\mathcal{A}}_{k_{\ast}}\ne\varnothing$.
Denoting $v_{\ast}=\pi(k_{\ast}+1)$, we distinguish two cases.

\smallskip
\noindent
{\bf Case 1: } $v_{\ast}$ is the lowest $\pi$-ranking vertex in its connected component. Thanks to Lemma~\ref{lem: prim}, we see that $\llbracket 1, k_{\ast}\rrbracket_{\pi}$ is disconnected from $\llbracket k_{\ast}+1, |V|\rrbracket_{\pi}$. On the other hand, from Algorithm~\ref{algo: graph} and the choice of $k_{\ast}$, we deduce that 
\begin{equation}
\label{id: 1-N}
\hat{\mathcal{A}}_{k_{\ast}}=
\bigcup_{i\le k_{\ast}}\mathcal N\big(\hat\pi(i))\setminus \{\hat\pi(1), \hat\pi(2), \dots, \hat\pi(k_{\ast})\} = \bigcup_{i\le k_{\ast}}\mathcal N\big(\pi(i))\setminus\llbracket 1, k_{\ast}\rrbracket_{\pi}. 
\end{equation}
It then follows from the connectivity property of $\llbracket 1, k_{\ast}\rrbracket_{\pi}$ that we would have $\hat{\mathcal{A}}_{k_{\ast}}=\varnothing$ in this case, which is impossible.

\smallskip
\noindent
{\bf Case 2: }
Let $k_-$ be the lowest $\pi$-rank of the connected component containing $v_{\ast}$ and we have $k_-<k_{\ast}+1$. Lemma~\ref{lem: prim} says that $\llbracket k_-, k_{\ast}+1\rrbracket_{\pi}$ is connected. Hence, 
we can find some $k_-\le k\le k_{\ast}$ so that $v_{\ast}\in \mathcal N(\pi(k))$. We note that~\eqref{id: 1-N} still holds and as a result, we find $v_{\ast}\in \hat{\mathcal{A}}_{k_{\ast}}$. According to the algorithm, we would then set $\hat\pi(k_{\ast}+1)=\pi(k_{\ast}+1)$, contradicting the definition of $k_{\ast}$. 

To conclude, it is impossible to have $k_{\ast}<|V|$, and therefore $\hat\pi=\pi$. 
\end{proof}

\subsection{A 2-neighbourhood exploration of $G(n_b, n_w, p)$}

Algorithm~\ref{algo: graph} allows us to explore $G(n_b, n_w, p)$ in Prim's order. Propositions~\ref{prop: prim} and~\ref{luka-preserver-graph-explo} combined then tell us that this amounts to examining the 1-neighhbourhood of each $\sigma(i)$ in turn. For our proof in the sublinear regime however, it is more convenient to explore the graph from vertices of the same colour. This leads to the following exploration algorithm. We will need the following notation. For a subset $A\subseteq V_n$, we define the following set to be its {\it 2-neighbourhood}:
\begin{equation}
\label{def: N2}
\mathcal N^2(A) = \mathcal N\big(\mathcal N(A)\big)\setminus A.
\end{equation}
We note in particular that $A$, $\cN(A)$ and $\cN^2(A)$ are disjoint from each other according to the definition. 
In the case that $A=\{v\}$, we denote $\cN^2(v)=\cN^2(\{v\})$. 

\begin{Alg}[2-neighbourhood exploration in Prim's order]
\label{algo: 2-explore}
Given $G(n_b, n_w, p)$ and the Prim sequence $(\sigma(i))_{i\in [n]}$, define an ordering on the set of black vertices $\sigma^b: [n_b]\to V^b_n$ as follows.

\smallskip
\noindent
{\bf Step 1. } Set $\sigma^b(1)$ to be the lowest ranking black vertex in $(\sigma(i))_{i\in [n]}$, and let 
$\cA^b_1=\cN^2(\sigma^b(1))$. 

\smallskip
\noindent
{\bf Step $\boldsymbol{2\le k\le n_b}$. } 
If $\cA^b_{k-1}$ is nonempty, set $\sigma^b(k)$ to be the lowest ranking element of $\cA^b_{k-1}$; otherwise, set $\sigma^b(k)$ to be the lowest ranking element of $V^b_n\setminus\{\sigma^b(1),\sigma^b(2),\dots, \sigma^b(k-1)\}$. 

Set $\cA^b_k=\cN^2(\{\sigma^b(1), \sigma^b(2),\dots, \sigma^b(k)\})$. 
\end{Alg}

The following proposition says that the ordering $\sigma^b$ is still consistent with Prim's order. Recall that $\Sigma(i)=\llbracket 1, i\rrbracket_{\sigma}$. 

\begin{Prop}
\label{prop: prim-explore}
For $k\in [n_b]$, let $\tau^b_k$ be the rank of the $k$-th black vertex in the Prim sequence, namely, $\tau^b_k=\min\{i: |\Sigma(i)\cap V^b_n|=k\}$. Then $\sigma^b(k)=\sigma(\tau^b_k)$ for each $k\in [n_b]$.     
\end{Prop}

\begin{proof}
Most of the arguments are similar to the ones used in the proof of Proposition~\ref{luka-preserver-graph-explo}; so we only highlight the differences. Suppose that $k_{\ast}=\max\{k: \sigma^b(k)=\sigma(\tau^b_k)\}<n_b$. The choice of $\sigma^b(k_{\ast}+1)$ means that we only need to consider the case where $\cA^b_{k_{\ast}}\ne\varnothing$. 
Denote $v_{\ast}=\sigma(\tau^b_{k_{\ast}+1})$. 

\smallskip
\noindent
{\bf Case 1: } $v_{\ast}$ is the lowest $\sigma$-ranking black vertex in its connected components.  Since connected components are intervals in Prim's order, we deduce that for each $k\le k_{\ast}$, $\sigma^b(k)=\sigma(\tau^b_k)$ is in a different connected component from $\{\sigma^b(k): k> k_{\ast}\}$. It follows that 
\[
 \cA^b_{k_{\ast}} = \bigcup_{k\le k_{\ast}}\mathcal N^2\big(\sigma^b(k)\big) \setminus \big\{\sigma^b(k): k\le k_{\ast}\big\} = \varnothing,
\]
which contradicts the previous assumption.

\smallskip
\noindent
{\bf Case 2: } $v_{\ast}$ is not the lowest $\sigma$-ranking black vertex in its connected components. Let us argue that $v_{\ast}\in \cA^b_{k_{\ast}}$. 
To that end, we first show that $I:=\{k\le k_{\ast}: v_{\ast}\in\cN^2(\sigma^b(k))\}$ is nonempty. 
Let $k_-$ be the lowest $\sigma$-rank of the connected component containing $v_{\ast}$. 
Then by Lemma~\ref{lem: prim} {\it (ii)}, there is a $\sigma$-good path $\rho$ from $\sigma(k_-)$ to $v_{\ast}$. By assumption, we can find another black vertex in the same component as $v_{\ast}$ that has a lower rank, say  $j\le \tau^b_{k_{\ast}}$. Again, there is a $\sigma$-good path $\rho'$ from $\sigma(k_-)$ to $\sigma(j)$. The concatenated path $\rho\cup \rho'$ consists of vertices with ranks $\le \tau^b_{k_{\ast}+1}$ and contains at least one black vertex distinct from $v_{\ast}$. So it must contain a black vertex at (graph) distance 2 from $v_{\ast}$ whose rank $j'\le \tau^b_{k_{\ast}}$. 
By the definition of $k_{\ast}$, we then have $\sigma(j')=\sigma^b(k)$ for some $k\le k_{\ast}$, which shows $I\ne \varnothing$. Since $v_{\ast}\notin\{\sigma^b(k): k\le k_{\ast}\}$, this implies that
\[
v_{\ast}\in \cA^b_{k_{\ast}}=\bigcup_{k\le k_{\ast}}\mathcal N^2\big(\sigma^b(k)\big) \setminus \big\{\sigma^b(k): k\le k_{\ast}\big\}.
\]
Step $k_{\ast}+1$ of the algorithm then asserts $\sigma^b(k_{\ast}+1)=v_{\ast}$, contradicting the definition of $k_{\ast}$. 
\end{proof}

\begin{Req}[Graph contraction and GEO]
We describe here an alternative proof of Proposition~\ref{prop: prim-explore}, based on the concept of graph contraction. If $G = (V, E)$ is a graph and $V' \subseteq V$,  
let $H=(V', E')$ be the graph on the vertex set $V'$ and the edge set $E'$ defined as follows:
\[
E'=\big\{ \{v, u\}: v\in V', u\in V'\setminus\{v\}, \exists\, W\subseteq V\setminus V' \text{ s.t. }\{v, u\}\cup W\text{ is connected in }G\big\}.
\]
We then say $H$ is the contraction of $G$ to $V'$. 
Suppose that $\pi$ is a GEO for $G$ and let $(v_k)_{1\le k\le |V'|}$ be the elements of $V'$ ranked in increasing $\pi$-order. Then $\pi': k\mapsto v_k$, $1\le k\le |V'|$, is an ordering on $V'$. 
One can show, that if $\pi$ is a GEO for $G$, then $\pi'$ is a GEO for $H$. Taking $G=G(n_b, n_w, p)$ and $V'=V^b_n$, we obtain a new graph $H^b$ on the vertex set $V^b_n$, where two black vertices are adjacent to each other if and only if they are adjacent to the same white vertex in $G$. The graph $H^b$ is also called the random intersection graph in the literature~\cite{KaScSC1999}. 
Since $\sigma$ is a GEO for $G(n_b, n_w, p)$, $k\mapsto \sigma(\tau^b_k)$, $k\in [n_b]$, is then a GEO for $H^b$. Finally, we point out that 
Algorithm~\ref{algo: 2-explore} is equivalent to Algorithm~\ref{algo: graph} for $H^b$ in $\sigma\circ\tau^b$-order. Proposition~\ref{luka-preserver-graph-explo} then yields Proposition~\ref{prop: prim-explore}. 
\end{Req}

\bigskip
Recall that 
$\Sigma^b(k) = \Sigma(k)\cap V^b_n$ and $\Sigma^w(k) = \Sigma(k) \cap V^w_n$. 
Instead of $\cA^b_k$, it is sometimes more convenient to study the following set
\[
\cO^b_k:=\cA^b_k\setminus \cA^b_{k-1}=\cN^2\big(\sigma^b(k)\big)\setminus \big(\cA^b_{k-1}\cup \Sigma^b(\tau^b_k)\big), \quad 1\le k\le n_b,
\]
with the convention that $\cA_0=\varnothing$. We will view $\cO^b_k$ as the 2-neighbourhood discovered at step $k$ in Algorithm~\ref{algo: 2-explore}. Note that by definition, $\cO^b_k$, $k\in [n_b]$, are disjoint. Let us introduce another representation of $\cO^b_k$. For $k\in [n_b]$, we consider 
\begin{equation}
\label{def: Kb}
\cK^b_k:=V^b_n\setminus \Big(\bigcup_{j\le k-1}\cO^b_j\,\bigcup \Sigma^b(\tau^b_k)\Big)
\end{equation}
to be the pool of black vertices available for discovery at step $k$; then we have 
\begin{equation}
\label{def: Obk}
\cO^b_k = \cN^2(\sigma^b(k))\cap \cK^b_k. 
\end{equation}
To see why~\eqref{def: Obk} is true, we claim that
\begin{equation}
\label{id: AB1}
\cA^b_{k-1}\cup\Sigma^b(\tau^b_k) = \bigcup_{j\in [k-1]}\cN^2(\sigma^b(j))\,\bigcup\,\Sigma^b(\tau^b_k)= \bigcup_{j\in [k-1]}\cO^b_j \,\bigcup\,\Sigma^b(\tau^b_k) = V^b_n\setminus \cK^b_k,
\end{equation}
Indeed, we have 
\begin{equation}
\label{id: Abk}
\cA^b_{k-1}=\cN^2\big(\Sigma^b(\tau^b_{k-1})\big) = \bigcup_{j\in [k-1]}\cN^2(\sigma^b(j))\,\setminus\,\Sigma^b(\tau^b_{k-1}).
\end{equation}
The first identity in~\eqref{id: AB1} follows. Then, 
\[
\cO^b_k = \cN^2(\sigma^b(k))\setminus\big(\cA^b_{k-1}\cup \Sigma^b(\tau^b_k)\big)=\cN^2(\sigma^b(k))\setminus\big( \bigcup_{j\in [k-1]}\cN^2(\sigma^b(j))\,\bigcup\, \Sigma^b(\tau^b_k)\big).
\]
Taking the union over $k$ yields the second identity in~\eqref{id: AB1}. The third one there now follows from~\eqref{def: Kb}. Finally, we deduce~\eqref{def: Obk} from~\eqref{id: AB1}. 

Next, let us define 
\[
\cK^w_1=V^w_n\setminus \Sigma^w(\tau^b_1);
\]
more generally for $2\le k\le n_b+1$, we set
\begin{equation}
\label{def: Ow}
\cO^w_{k-1} = \cN\big(\sigma^b(k-1)\big)\cap \mathcal K^w_{k-1}\quad\text{and}\quad \mathcal K^w_k=\mathcal K^w_{k-1}\setminus\big( \cO^w_{k-1}\cup \Sigma^w(\tau^b_k)\big).
\end{equation}
In analogue to $\cK^b_k$ and $\cO^b_k$, $\mathcal K^w_{k}$ is the pool of available white vertices at step $k$  and $\cO^w_k$ the discovered 1-neighbourhood of $\sigma^b(k)$, which consists of white vertices only. 
The sets $\cO^w_k, k\in [n_b]$, are also disjoint. 

\section{Bond percolation of the complete bipartite graph}
\label{sec: percolation}

As indicated in Proposition~\ref{prop: prim}, there is a close connection between the Prim sequence and the bond percolation of $K_{n_b, n_w}$. We collect here various properties of the bond percolation that are useful for our proofs. Broadly speaking, the results below mirror those for the bond percolation of the complete graph $K_n$, or equivalently the graph process $(G(n, p))_{p\in [0, 1]}$.

We will use the following notation for probability distributions. 
For $n\in \N$ and $p \in [0, 1]$, we denote by $\cB(n, p)$ the Binomial distribution with parameters $n$ and $p$. It will be convenient to write $\cB(0, p)$ for the law of the degenerate random variable $X\equiv 0$. We also denote by Poisson$(c)$ the Poisson distribution with mean $c\in (0, \infty)$. 

\subsection{A two-type branching process}

Let $p\in (0, 1)$ and recall the graph $G(n_b, n_w, p)=(V_n, E_n(p))$ with the edge set $E_n(p)$ from~\eqref{def: Enp}. It is not difficult to see that each black vertex in $G(n_b, n_w, p)$ is adjacent to a random number of white vertices with a $\cB(n_w, p)$ distribution, and each white vertex is adjacent to a random number of black vertices with a $\cB(n_b, p)$ distribution. 
From now on, we focus on the sparse regime: fix $\lambda\in (0, \infty)$ and let
\begin{equation}
\label{hyp: p}
p= \frac{\lambda}{\sqrt{n_b n_w}}, \quad n\ge 1.
\end{equation}
It turns out that the local limit of $G(n_b, n_w, p)$, as $n\to\infty$, is closely connected to the following two-type branching process. Let $\{\mathbf Z_k=(Z^{(1)}_k, Z^{(2)}_k): k\ge 0\}$ be a $\Z_+^2$-valued Markov chain that evolves as follows.
For all $k\ge 1$, given $\mathbf Z_k=(z_1, z_2)\in \Z_+^2$, we have 
\[
Z^{(1)}_{k+1} = \sum_{i=1}^{z_2} Y^{(k, 1)}_i, \quad Z^{(2)}_{k+1} = \sum_{i=1}^{z_1} Y^{(k, 2)}_i,
\]
where $(Y^{(k, 1)}_i)_{i\ge 1}, (Y^{(k, 2)}_i)_{i\ge 1}$ are two independent families of i.i.d.~random variables with 
\[
Y^{(k, 1)}_1\sim \text{ Poisson} (\lambda\gamma_{\theta}^{-1}), \quad Y^{(k, 2)}_1\sim \text{ Poisson}(\lambda\gamma_{\theta}),
\] 
where $\gamma_{\theta}=\sqrt{(1-\theta)/\theta}$. 
For $i, j\in\{1, 2\}$, denote $m_{i, j}=\mathbb E[Z^{(j)}_1\,|\, Z^{(i)}_0=1]$. Straightforward computations show that 
\[
M:=
\begin{pmatrix}
m_{1,1} & m_{1, 2}\\
m_{2, 1} & m_{2, 2}
\end{pmatrix}
=\begin{pmatrix}
0 & \lambda\gamma_{\theta}\\
\lambda\gamma_{\theta}^{-1} & 0
\end{pmatrix}.
\]
has eigenvalues $\lambda, -\lambda$ and $(1, \gamma_{\theta})$ is a left eigenvector for $\lambda$. 
The extinction probabilities are defined as
\begin{equation}
\label{def: qi}
q_1(\lambda) = \mathbb P\big(\exists\, k\ge 1: \mathbf Z_k=(0, 0) \,\big|\, \mathbf Z_0=(1, 0)\big), \quad q_2(\lambda) = \mathbb P\big(\exists\, k\ge 1: \mathbf Z_k=(0, 0) \,\big|\, \mathbf Z_0=(0, 1)\big). 
\end{equation}

\begin{The}
\label{thm: br}
Let $\gamma_{\theta}=\sqrt{\frac{1-\theta}{\theta}}$. 
The extinction probabilities are given as follows.
\begin{enumerate}[(i)]
\item
(Sub)critical regime: if $\lambda\le 1$, then $q_1(\lambda)=q_2(\lambda)=1$. 
\item
Supercritical regime: if $\lambda >1$, then $(q_1(\lambda), q_2(\lambda))$ is the unique solution in $(0, 1)^2$ of the fixed-point equation:
\begin{equation}
\label{eq: fixedpt}
\left\{
\begin{array}{l}
x = e^{\lambda \gamma_{\theta}(y - 1)} \\
y = e^{\lambda \gamma_{\theta}^{-1}(x - 1)}
\end{array}\right..
\end{equation}
Furthermore, the only other solution of~\eqref{eq: fixedpt} is $(1, 1)$. 
\end{enumerate}
\end{The}

\begin{proof}
As $(\mathbf Z_k)_{k\ge 0}$ is 2-periodic, classic results on the multi-type branching processes such as Theorem 2 in Chapter V.2~\cite{AN_book} do not apply directly. However, we note that $(Z^{(1)}_{2k})_{k\ge 0}$ is an ordinary branching process where the offspring distribution is given by Poisson$(Y\cdot \lambda \gamma_{\theta}^{-1})$ with $Y\sim$ Poisson$(\lambda\gamma_{\theta})$. 
Moreover, this branching process is supercritical if and only if $\lambda>1$. 
Denote its extinction probability as
\[
\tilde q_1(\lambda):=\mathbb P\big(\exists\,k\ge 1: Z^{(1)}_{2k}=0\,\big|\,\mathbf Z_0=(1,0)\big).
\] 
Then for each $\lambda>0$,  $\tilde q_1(\lambda)$ is the smallest solution of the equation:
\begin{equation}
\label{eq: fixpt'}
s=\exp\Big\{\lambda \gamma_{\theta}\Big(e^{\lambda \gamma_{\theta}^{-1}(s-1)}-1\Big)\Big\}.
\end{equation}
Criticality of the branching process then implies that for $\lambda\in (0, 1]$,  the equation has a unique solution with $\tilde q_1(\lambda)=1$, while for $\lambda>1$, the equation has two solutions, $1$ and 
$\tilde q_1(\lambda)\in (0, 1)$. Set 
\[
\tilde q_2(\lambda) = \exp\big\{\lambda\gamma_{\theta}^{-1}(\tilde q_1(\lambda)-1)\big\};
\]
then $(\tilde q_1(\lambda), \tilde q_2(\lambda))$ solves~\eqref{eq: fixedpt}.  On the other hand, if $(\hat q_1(\lambda), \hat q_2(\lambda))$ is a solution of~\eqref{eq: fixedpt}, then necessarily $\hat q_1(\lambda)$ solves~\eqref{eq: fixpt'}. We further have $\hat q_1(\lambda)\in (0, 1)$ if and only if $\hat q_2(\lambda)\in (0, 1)$. We deduce that for $\lambda\in (0, 1]$,~\eqref{eq: fixedpt} has a unique solution $(1, 1)$, while for $\lambda>1$,~\eqref{eq: fixedpt} has two sets of solutions, given by $(1, 1)$ and $(\tilde q_1(\lambda), \tilde q_2(\lambda))$. 

It remains to show that $q_i(\lambda)=\tilde q_i(\lambda)$, $i\in \{1, 2\}$. Indeed, we note that $(Z^{(2)}_{2k})_{k\ge 0}$ is also a branching process. It follows that $Z^{(2)}_0=0$ implies $Z^{(0)}_{2k}=0$ for all $k\ge 0$. Comparing the definitions of $q_1(\lambda)$ with $\tilde q_1(\lambda)$, we see that $\tilde q_1(\lambda)\le q_1(\lambda)$. On the other hand,  $\mathbf Z_k=(0, 0)$ implies $\mathbf Z_{k'}=(0, 0)$ for all $k'\ge k$; hence $Z_{2\ell}^{(1)}=0$ for all $2\ell\ge k$. We then find $q_1(\lambda)=\tilde q_1(\lambda)$, which shows $\tilde q_1(\lambda)=q_1(\lambda)$. 
The case of $q_2(\lambda)$ can be similarly argued. 
\end{proof}

\subsection{Size of the giant component}

The {\it size} of a connected component $\mathcal C$ of $G(n_b, n_w, p)$ refers to the number of vertices contained in $\mathcal C$. Let $\mathcal C_n(\lambda)$ be the largest connected component of $G(n_b, n_w, p)$ with $p=\lambda/\sqrt{n_b n_w}$, breaking ties arbitrarily. 
In parallel to the classical results on the Erd\H{o}s--R\'enyi graph, the random bipartite graph $G(n_b, n_w, p)$ also exhibits a phase transition in the size of $\mathcal C_n(\lambda)$ as $n\to\infty$, with a giant component emerging right after $\lambda=1$, matching the critical point of the previous branching process $\{\mathbf Z_k: k\ge 0\}$ (~\cite{Johansson12}). 
Denote by $C^{\mathrm b}_n(\lambda)$ and $C^{\mathrm w}_n(\lambda)$ the respective numbers of black and white vertices in $\mathcal C_n(\lambda)$. The following result on the limit of $(C^{\mathrm b}_n(\lambda), C^{\mathrm w}_n(\lambda))$ is a key ingredient of our proof in the linear regime. 

\begin{The}[\cite{Johansson12}, Theorem 12]
\label{thm: giant-size}
Assume~\eqref{hyp: theta} and~\eqref{hyp: p} and let $\lambda>1$. Recall $q_1(\lambda), q_2(\lambda)$ from~\eqref{def: qi}. The following convergences in probability hold as $n\to\infty$: 
\[
\frac{C^{\mathrm b}_{n}(\lambda)}{n} \to \theta\big(1-q_1(\lambda)\big) \quad \text{and}\quad \frac{C^{\mathrm w}_n(\lambda)}{n} \to (1-\theta)\big(1-q_2(\lambda)\big). 
\]
\end{The}

We have seen in Theorem~\ref{thm: br} that $q_i(\lambda), i=1, 2$, are solutions of a fixed-point equation. This allows us to deduce the following properties. 
\begin{Lem}
\label{lem: solv'}
Let $\gamma_{\theta}=\sqrt{\frac{1-\theta}{\theta}}$. The following statements are true. 
\begin{enumerate}[(i)]
\item
For $i\in \{1, 2\}$, $q_i(\lambda)$ is differentiable and strictly decreasing on $(1, \infty)$; moreover, $q_i(1+)=1$ and  $q_i(\infty)=0$.
\item
As $\lambda\to 1+$, we have $q_1(\lambda)-1 = \gamma_{\theta} (q_2(\lambda)-1) + o\big(|q_2(\lambda)-1|\big)$, 
\end{enumerate} 
\end{Lem}

\begin{proof}
For (i), we will only prove the properties for $q_1(\lambda)$; the case of $q_2(\lambda)$ can be similarly argued. Recall that for $\lambda>1$, $q_1(\lambda)$ is the unique solution in $[0, 1)$ of the fixed-point equation~\eqref{eq: fixpt'}. Let us introduce
\begin{equation}
\label{def: F}
F(x, \lambda) = \lambda \gamma_{\theta}\Big(e^{\lambda \gamma_{\theta}^{-1}(x-1)}-1\Big)-\log x, \quad x>0, \; \lambda>1.
\end{equation}
Then $F(q_1(\lambda), \lambda)=0$. 
We note that $x\mapsto F(x, \lambda)$ is convex with $F(0+, \lambda)=\infty$ and $F(1, \lambda)=0$. Since $\lambda>1$, we have $q_1(\lambda)\in (0, 1)$. Therefore, for each $\lambda>1$, $F(x, \lambda)$ has two roots with $q_1(\lambda)$ being the smaller one. It follows $\partial_x F(q_1(\lambda), \lambda)<0$. By the implicit function theorem, $q_1(\lambda)$ is differential in $\lambda\in (1, \infty)$ and we further have
\[
q'_1(\lambda) = -\frac{\partial_{\lambda} F\big(q_1(\lambda), \lambda\big)}{\partial_x F\big(q_1(\lambda), \lambda\big)} = -\frac{\gamma_{\theta}(e^{\lambda \gamma_{\theta}^{-1}(q_1(\lambda)-1)}-1)+\lambda (q_1(\lambda)-1)e^{\lambda \gamma_{\theta}^{-1}(q_1(\lambda)-1)}}{\partial_x F\big(q_1(\lambda),\lambda\big)}<0,
\]
for all $\lambda>1$. This proves the monotonicity. In particular, $q_1(\infty):=\lim_{\lambda\to \infty}q_1(\lambda)$ exists and is a real number in $[0, 1)$. Meanwhile, we have
\begin{equation}
\label{eq: ft-log}
\log\big(q_1(\lambda)\big) = \lambda \gamma_{\theta}\Big(e^{\lambda \gamma_{\theta}^{-1}(q_1(\lambda)-1)}-1\Big).
\end{equation}
By the monotone convergence theorem, the right-hand side above tends to $-\infty$ as $\lambda\to \infty$, which implies $q_1(\infty)=0$. 
Next, taking $\lambda\to 1+$ on the both sides of~\eqref{eq: ft-log} leads to 
\[
\log\big(q_1(1+)\big) = \gamma_{\theta}\Big(e^{\gamma_{\theta}^{-1}(q_1(1+)-1)}-1\Big).
\]
Elementary arguments show that the function $F(x, 1)$
has a unique root at 1. Hence, $q_1(1+)=1$, which concludes the proof of {\it (i)}. 

We have seen in {\it (i)} that both $q_i$ are right-continuous at $\lambda=1$. Since they satisfy~\eqref{eq: fixedpt}, combining the first equation with the expansion $e^x=1+x+O(x^2)$ as $x\to 0$ implies {\it (ii)}.

\end{proof}

Let us introduce 
\begin{equation}
\label{def: ell}
\ell(\lambda) = \theta\big(1-q_1(\lambda)\big) + (1-\theta)\big(1-q_2(\lambda)\big), \quad \rho(\lambda) = \frac{\theta\big(1-q_1(\lambda)\big)}{\ell(\lambda)}.
\end{equation}
\begin{Lem}
\label{lem: solv}
The following statements hold true.
\begin{enumerate}[(i)]
\item
For each $\lambda>1$, $(\ell(\lambda), \rho(\lambda))$ is the unique solution in $(0, \infty)^2$ of~\eqref{eq: fixpt-thm}; moreover, $\ell(\lambda)\in (0, 1)$ and $\rho(\lambda)\in (0, 1)$. 
\item
The map $\lambda\mapsto \ell(\lambda)$ is differentiable and strictly increasing on $(1, \infty)$. Moreover, $\lim_{\lambda\to \infty} \ell(\lambda)=1$ and $\lim_{\lambda\to 1+}\ell(\lambda) = 0$. As a result, $\ell: (1, \infty)\to (0, 1)$ is a bijection. 
\item
As $\lambda\to 1+$, we have $ \rho(\lambda) \to (1+\gamma_{\theta})^{-1}$. 
\end{enumerate}
\end{Lem}

\begin{proof}
For {\it (i)}, we note that if $(\tilde\ell(\lambda), \tilde \rho(\lambda))$ is a solution of~\eqref{eq: fixpt-thm}, then 
\[
\tilde q_1(\lambda) = 1-\frac{\tilde \ell(\lambda) \cdot \tilde \rho(\lambda)}{\theta}, \quad \tilde q_2(\lambda) =  1-\frac{\tilde \ell(\lambda) \cdot \big(1-\tilde\rho(\lambda)\big)}{1-\theta}
\]
solves~\eqref{eq: fixedpt}. Since $\lambda>1$, by Theorem~\ref{thm: br}, the solutions of~\eqref{eq: fixedpt} are given by $(1, 1)$ and $(q_1(\lambda), q_2(\lambda))$. It follows that $(\ell(\lambda), \rho(\lambda))$ is the unique solution in $(0, \infty)^2$. We have also seen in Theorem~\ref{thm: br} that $q_i(\lambda)\in (0, 1)$, $i\in \{1, 2\}$. Hence, $\ell(\lambda)\in (0, 1)$ and $\rho(\lambda)\in (0, 1)$. 
Properties in {\it(ii)} and {\it(iii)} are immediate consequences of Lemma~\ref{lem: solv'}.
\end{proof}

\section{Proof in the linear regime}
\label{sec: linear}

We prove here Theorem~\ref{thm: linear}. Recall the random graph $G(n_b, n_w, p)$ and assume that $p$ is as in~\eqref{hyp: p}. 
Recall from Theorem~\ref{thm: giant-size} that for $\lambda>1$, $G(n_b, n_w, p)$ contains a giant component $\mathcal C_n(\lambda)$ with high probability.  
Let us denote by $k_-(\lambda)$ (resp.~$k_+(\lambda)$) the lowest (resp.~highest) rank of the vertices in $\mathcal C_n(\lambda)$. 
Thanks to Proposition~\ref{prop: prim}, the vertex set of $\mathcal C_n(\lambda)$ is precisely $\{\sigma(i): k_-(\lambda)\le i\le k_+(\lambda)\}$. 

\begin{Lem}
\label{lem: kminus}
For each $\lambda\in (1, \infty)$, the following convergence in probability takes place:
\[
\frac{k_-(\lambda)}{n}\xrightarrow{n\to\infty} 0.
\]
\end{Lem}

\begin{proof}
Recall that $C^{\mathrm b}_n(\lambda)$ and $C^{\mathrm w}_n(\lambda)$ are the respective numbers of black and white vertices in $\mathcal C_n(\lambda)$. 
Take $v\in V_n$. Let us first show the following:
\begin{equation}
\label{id: unif_giant}
\mathbb P(v\in \mathcal C_n(\lambda)\,|\, C^{\mathrm b}_n(\lambda), C^{\mathrm w}_n(\lambda)) \ge \frac{C^{\mathrm b}_n(\lambda)}{n_b}\mathbf 1_{\{v\in V^b_n\}} + \frac{C^{\mathrm w}_n(\lambda)}{n_w}\mathbf 1_{\{v\in V^w_n\}}.
\end{equation}
To see why this is true, we first note that the probability of each configuration of $G(n_b, n_w, p)$ only depends on the number of edges present in the configuration.  
Denote by $\Pi_{n}$ the set of all permutations $\pi: V_n\to V_n$ satisfying $\pi(V^b_n)=V^b_n$. 
Let $\varrho$ be a uniform element of $\Pi_{n}$. 
Let $c_b, c_w\in \N$ and let $\Gamma_n$ be the set of all  graphs of $n_b$ black vertices, $n_w$ white vertices, and containing a largest connected component with $c_b$ black vertices and $c_w$ white vertices. 
For $g\in \Gamma_n$, denote by $A(g; v)$ the event that $v$ belongs to a largest connected component of $g$. 
If $v$ is black and $g$ has a unique largest connected component, then  by summing over all the possible permutations, we have 
\[
\mathbb P\big(A(\varrho(g); v)\big) = \sum_{\pi\in\Pi_n}\frac{1}{n_b!n_w!}\mathbf 1_{A(\pi(g); v)} = \frac{c_b}{n_b}.
\]
This probability is even larger if $g$ has more than one largest connected components. 
We have

\[
\mathbb P\big(v\in \mathcal C_n(\lambda), C^{\mathrm b}_n(\lambda)  = c_b, C^{\mathrm w}_n(\lambda) = c_w\big) = \sum_{g\in \Gamma_n} \mathbb P(G(n_b, n_w, p)=g)\mathbf 1_{A(g;v)}.
\]
On the other hand, noting that the law of $G(n_b, n_w, p)$ is invariant by $\varrho$, we deduce that the previous probability can also be written as follows:
\begin{align*}
    \mathbb P\big(v\in \mathcal C_n(\lambda); C^{\mathrm b}_n(\lambda)  = C^{\mathrm b}_n, C^{\mathrm w}_n(\lambda) = c_w\big) 
& = \sum_{g\in \Gamma_n} \mathbb P\big(G(n_b, n_w, p)=g\big)\cdot\mathbb P\big(A(\varrho(g); v)\big)\\
& \ge\frac{c_b}{n_b} \sum_{g\in \Gamma_n} \mathbb P\big(G(n_b, n_w, p)=g\big)\\
& = \frac{c_b}{n_b}\cdot \mathbb P\big(C^{\mathrm b}_n(\lambda)  = c_b, C^{\mathrm w}_n(\lambda) = c_w\big).
\end{align*}
The case of a white vertex can be argued similarly. Combing the two cases yields~\eqref{id: unif_giant}. 

We say that a vertex is a lead if it is the lowest ranking vertex in its connected component of $G(n_b, n_w, p)$. 
For $\ell\ge 1$, let $\zeta_{\ell}$ be the rank of the $\ell$-th lead vertex and let $\cH_{\ell}$ be the sigma-algebra as defined in Section~\ref{sec: app}.  Let $N^b_{\ell}$ and $N^w_{\ell}$ stand for the number of black and white vertices among $\{\sigma(i): i\ge \zeta_{\ell}\}$. We note that $\zeta_{\ell}, \sigma(\zeta_{\ell}), N^b_{\ell}$ and $N^w_{\ell}$ are all $\cH_{\ell}$-measurable. Denote by $\hat K$ the induced subgraph of $K_{n_b, n_w}$ on $\{\sigma(i): i\ge \zeta_{\ell}\}$ and by $\hat G$ the induced subgraph of $G(n_b, n_w, p)$ on the same vertex set. 
Thanks to Lemma~\ref{lem: Hk} {\it(ii)}, $\hat K$ is a complete bipartite graph with $N^b_{\ell}$ black vertices, $N^w_{\ell}$ white vertices and i.i.d.~uniform edge weights. 
Suppose that $\mathcal C_n(\lambda)$ is the $L$-th connected component in Prim's rank. We note that given $\cH_{\ell}$ and $\{L\ge \ell\}$, we have
\[
\{L=\ell\} = \{k_-(\lambda)=\zeta_{\ell}\}=\{\sigma(\zeta_{\ell})\in\mathcal C_n(\lambda)\}.
\]
On the event $L\ge \ell$, a largest connected component of $G(n_b, n_w, p)$ is also a largest connected component of $\hat G$. Noting that $N^b_{\ell}\le n_b, N^w_{\ell}\le n_w$, we apply~\eqref{id: unif_giant} to $\hat G$ and $\sigma(\zeta_{\ell})$ to find that
\[
\mathbb P\big(L=\ell\,\big|\, L\ge \ell, \cH_{\ell}, C^{\mathrm b}_n(\lambda), C^{\mathrm w}_n(\lambda)\big) \ge \min\Big(\frac{C^{\mathrm b}_n(\lambda)}{n_b}, \frac{C^{\mathrm w}_n(\lambda)}{n_w}\Big)=:\xi_n. 
\]
It follows that 
\[
\mathbb P\big(L>\ell\,\big|\, \xi_n\big)=\prod_{1\le \ell'\le \ell} \mathbb P\big(L> \ell'\,\big|\, L\ge \ell', \xi_n\big) \le (1-\xi_n)^{\ell}. 
\]
This shows that $L$ is stochastically bounded by a geometric random variable of mean $\xi_n^{-1}$. Thanks to Theorem~\ref{thm: giant-size}, $\xi_n$ converges in probability to a strictly positive real number.  Consequently, the distribution of $L$ is tight. Let $C^{(2)}(\lambda)$ be the size of the second largest connected component in $G(n_b, n_w, p)$. Theorem 12 in \cite{Johansson12} shows that $C^{(2)}(\lambda) = o_p(n)$ as $n\to\infty$. The conclusion now follows from the bound that $k_-(\lambda) \le  L \cdot C^{(2)}(\lambda)$.
\end{proof}

We are now ready for the main proof in the linear regime.

\begin{proof}[Proof of Theorem~\ref{thm: linear}]
Thanks to Lemma~\ref{lem: solv}, it only remains to prove the convergence in~\eqref{eq: thm2}. 
We have seen in Lemma~\ref{lem: solv} that $\ell: (1, \infty)\to (0, 1)$ is invertible. Fix $s\in (0,1)$ and take $\lambda_{\ast}=\ell^{-1}(s)$, so that $\ell(\lambda_{\ast})=s$. 
We note that the size of the giant component is given by $k_+(\lambda)-k_-(\lambda)+1$. Theorem~\ref{thm: giant-size} then says that for each $\lambda>1$, 
\[
\frac{1}{n}\big(k_+(\lambda)-k_-(\lambda)\big)\xrightarrow{n\to\infty} \ell(\lambda)\quad\text{in probability}.
\]
Combining this with Lemma~\ref{lem: kminus}, we find  $
k_+\big(\lambda_{\ast}\big) = sn +o_p(n)$, which then implies
\[
\big|\Sigma^b(\lfloor sn\rfloor)\big| = \Big|\Sigma^b\big(k_+(\lambda_{\ast})\big)\Big|+o_p(n), \quad n\to\infty.
\]
Lemma~\ref{lem: kminus} also yields that $|\Sigma^b(k_-(\lambda_{\ast}))| \le k_-(\lambda_{\ast})=o_p(n)$. Note that $|\Sigma^b(k_+(\lambda_{\ast}))|-|\Sigma^b(k_-(\lambda_{\ast}))|+1$ corresponds to the number of black vertices in the giant component. 
We deduce that as $n\to\infty$, 
\begin{align*}
\rho^{(n)}_{\lfloor sn\rfloor} & = \frac{1}{sn}\Big(\Big|\Sigma^b(\lfloor sn\rfloor)\Big|+o_p(n)\Big)= \frac{1}{sn}\Big(\Big|\Sigma^b\big(k_+(\lambda_{\ast})\big)\Big|+o_p(n)\Big) \\
& = \frac{1}{sn}\Big(\Big|\Sigma^b\big(k_+(\lambda_{\ast})\big)\Big|-\Big|\Sigma^b\big(k_-(\lambda_{\ast})\big)\Big|+o_p(n)\Big)
= \frac{1}{sn}\Big(C^{\mathrm b}_n(\lambda_{\ast})+o_p(n)\Big). 
\end{align*}
The conclusion now follows from Theorem~\ref{thm: giant-size} and~\eqref{def: ell}. 
\end{proof}

\section{Distribution of the neighbourhoods}
\label{sec: neighbour}

In preparation for the proof of Theorem~\ref{thm: sublinear}, we study here the sizes of the neighbourhoods discovered in Algorithm~\ref{algo: 2-explore}, which explores the black vertices in the random graph $G(n_b, n_w, p)$ in Prim's ranking. Recall from~\eqref{def: Kb}-\eqref{def: Ow} the sets $\cK^b_k$, $\cO^b_k$, $\cK^w_k$ and $\cO^w_k$. We denote
\begin{equation}
\label{def: O-size}
K^b_k:=|\cK^b_k|, \quad K^w_k:=|\cK^w_k|, \quad O^b_k:=|\cO^b_k|\quad\text{and}\quad O^w_k:=|\cO^w_k|. 
\end{equation}
We point out that all the previous random variables depend on the graph $G(n_b, n_w, p)$, and thus on $p\in [0, 1]$; we nevertheless suppress this dependency to ease the notation. 
In the main result of this section (Proposition~\ref{prop: incre} below), we determine the respective distributions of $O^b_k$ and $O^w_k$. We then derive upper and lower bounds for these random variables, which are key ingredients for the proofs in Section~\ref{sec: sublinear}.  

Towards that end, we first introduce some notions. 
We recall that a vertex of $G(n_b, n_w, p)$ is called a {\it lead} if it has the lowest Prim rank in its connected component. If a component contains at least one black vertex, then the black vertex with the lowest rank will be referred to as a {\em root vertex}, as it is the root of the spanning tree extracted by Algorithm~\ref{algo: 2-explore}. Clearly, a black lead vertex is also a root vertex, but not necessarily vice versa. 
Recall $\Sigma^b(k)$ and $\Sigma^w(k)$ from~\eqref{def: Sigma-set} and $\tau^b_k$ from Proposition~\ref{prop: prim-explore}. Note that in particular, we have $\Sigma^b(\tau^b_k)=\{\sigma(\tau^b_i): i\in [k]\}=\{\sigma^b(i): i\in [k]\}$. 
Let $\cJ^w_k$ denote the set of all white lead vertices from $\Sigma^w(\tau^b_k)$. We also denote $\cI^b_k$ to be the set of all root vertices from $\Sigma^b(\tau^b_k)$. Recall that for $A\subseteq V_n$, $\cN(A)$ and $\cN^2(A)$ are the respective 1-and 2-neighbourhoods of $A$ in $G(n_b, n_w, p)$, as defined in~\eqref{def: N1} and~\eqref{def: N2}. 
 
Finally, recall from Algorithm~\ref{algo: 2-explore} that $\cA^b_k=\cN^2(\Sigma^b(\tau^b_k))$. 

\begin{Lem}
\label{lem: Kset}
Let $p\in [0, 1]$. The following statements hold true for $G(n_b, n_w, p)$.
\begin{enumerate}[(i)]
\item
Let $k\in [n_b]$. If $v\in \cJ^w_k$ and $\cN(v)\ne \varnothing$, then we can find some $j\in [k]$ so that $v=\sigma(\tau^b_j-1)$, $\sigma(\tau^b_j)\in \cN(v)$, and $\sigma(\tau^b_j)\in \cI^b_k$. 
\item
For all $k \in [n_b]$, $\cJ^w_{n_b}\cap \cO^w_k=\varnothing$ and $\cI^b_{n_b}\cap \cO^b_k=\varnothing$.
\item
For all $k\in [n_b]$, we have the following compositions into disjoint sets:
\begin{equation}
\label{id: Kset}
V^w_n \setminus \cK^w_k = \bigcup_{j\in[k-1]}\cO^w_j\,\bigcup\,\cJ^w_k \quad\text{and}\quad V^b_n\setminus \cK^b_k = \bigcup_{j\in [k-1]}\cO^b_j \,\bigcup\, \cI^b_k.
\end{equation}
\item
For all $k\in [n_b]$, we also have
\begin{equation}
\label{eq: Nset}
\cN\big(\Sigma^b(\tau^b_k)\big) \subseteq \bigcup_{j\in [k]}\cO^w_j\,\bigcup\,\cJ^w_k \quad\text{and}\quad \cA^b_k=\cN^2\big(\Sigma^b(\tau^b_k)\big) = \bigcup_{j\in [k]}\cO^b_j\setminus\Sigma^b(\tau^b_k).
\end{equation}
\end{enumerate}
\end{Lem}

\begin{proof}

For {\it(i)}, we can assume that $v=\sigma(i)$ for some $i<\tau^b_k$. Since $v$ is a lead vertex with a nonempty neighbourhood, we deduce from Proposition~\ref{lem: prim} and the definition of GEO that $\llbracket i, i+1\rrbracket_{\sigma}$ is connected in $G(n_b, n_w, p)$. 
So it must be the case that $\sigma(i+1)\in \cN(v)$. As $v$ is a white vertex, $\sigma(i+1)$ is then a black vertex. Therefore, we have $i+1=\tau^b_j$ for some $j\le k$. It is also clear that $\sigma(\tau^b_j)$ is the lowest ranking black vertex in its component. 

\smallskip
For {\it(ii)}, let us assume that $v=\sigma(i)\in \cJ^w_{n_b}$. 
As $v$ is a lead vertex, it is not connected to any vertex of a lower rank. Hence, $v\notin \cO^w_k$ for any $\tau^b_k<i$, since $\cO^w_k\subseteq \cN(\sigma^b(k))$ by definition. If $\tau^b_k > i$, we note that $v\in \Sigma^w(\tau^b_k)$, so that $v\notin \cK^w_k$ and therefore $v\notin \cO^w_k$. This proves that  $\cJ^w_{n_b}$ is disjoint from any $\cO^w_k$. 

Next, we show that $\cI^b_{n_b}\cap \cO^b_k=\varnothing$. It suffices to show that if $u=\sigma^b(j)\in \cI^b_j$, then $u\notin \cO^b_k$ for all $k\in [n_b]$. First, this is clear if $k\ge j$, since in that case we will have $u\in \Sigma^b(\tau^b_k)\subseteq V^b_n\setminus \cK^b_k$. Second, consider the case where $k<j$. However, as $u$ is a root vertex, it is in a different connected component from  any black vertex of a lower rank. It follows that $u\notin \cN^2(\sigma^b_k)$ and therefore $u\notin \cO^b_k$ for all $k<j$. 

\smallskip
Proceeding to the proof of {\it (iii)}, we first deduce from the  definition that $V^w_n\setminus \cK^w_k=\cup_{j\le k-1}\cO^w_j\cup \Sigma^w(\tau^b_k)$. We have also seen in {\it (ii)} that $\cJ^w_k\subseteq \Sigma^w(\tau^b_k)\setminus \cup_{j\le k-1}\cO^w_j$. So it suffices to show that for all $k\in [n_b]$, 
\begin{equation}
\label{id: Kset'}
\Sigma^w(\tau^b_k)\setminus \Big(\bigcup_{j\le k-1}\cO^w_j\Big) \subseteq \cJ^w_k. 
\end{equation}
We will prove this with an induction on $k$. For $k=1$, we note that $\tau^b_1\in \{1, 2\}$, since $\sigma(\tau^b_1)$ is necessarily an endpoint of the first Prim edge $e_1$. In either case, we readily find that  $\Sigma^w(\tau^b_1)= \cJ^w_1$. 

For the induction step,  assuming that the identity holds for some $k-1\ge 1$, let us argue that it also holds for $k$, which reduces to showing that 
\begin{equation}
\label{id: Kset''}
\Sigma^w(\tau^b_k)\setminus \Big(\Sigma^w(\tau^b_{k-1})\bigcup \bigcup_{j\le k-1}\cO^w_j\Big) \subseteq \cJ^w_k.
\end{equation}
Let us denote $A_k=\Sigma^w(\tau^b_k)\setminus\Sigma^w(\tau^b_{k-1})=\{\sigma(i): \tau^b_{k-1}<i< \tau^b_k\}$. Assume $A_k$ to be nonempty; otherwise~\eqref{id: Kset''} is trivial. We consider the following cases.

\smallskip
\noindent
{\bf Case I:} $\sigma^b(k)$ is in the same connected component as $\sigma^b(k-1)$. In that case, by Proposition~\ref{prop: prim}, $A_k$ is contained in the connected component containing $\sigma^b(k)$. Let us show in this case $A_k\subseteq \cup_{j\le k-1}\cO^w_j$, which will then imply ~\eqref{id: Kset''}. Indeed, let $\sigma(i)\in A_k$; thanks to Lemma~\ref{lem: prim}, there is a $\sigma$-good path from the lead vertex of its connected component to $\sigma(i)$. Since $\sigma(i)$ is white and not the lead vertex, it must be adjacent to a black vertex of rank at most $\tau^b_{k-1}$. 
Let $j_{\ast}$ be the smallest such $j$ that $\sigma(i)\in \cN(\sigma^b(j))$. Then we have $j_{\ast}\le k-1$ and $\sigma(i)\notin \cup_{j'<j_{\ast}}\cN(\sigma^b(j'))$. It follows that $\sigma(i)\notin \cup_{j'<j_{\ast}}\cO^w_{j'}$, which is contained in $\cup_{j'<j_{\ast}}\cN(\sigma^b(j'))$. Moreover, as $i>\tau^b_{k-1}$, we have $\sigma(i)\notin \Sigma^w(\tau^b_{j_{\ast}})$. We then deduce that $i\in \cK^w_{j_{\ast}}$ and so $i\in \cO^w_{j_{\ast}}$. This shows $A_k\subseteq \cup_{j\le k-1}\cO^w_j$. 

\smallskip
\noindent
{\bf Case II:} $\sigma^b(k)$ is in a different connected component from $\sigma^b(k-1)$. In that case, let $i_{\ast}$ be the highest rank of the vertices in the connected component containing $\sigma^b(k-1)$. Also, let $i^{\ast}$ be the lowest rank of the vertices in the connected component containing $\sigma^b(k)$. 
By Proposition~\ref{prop: prim}, we have 
\[
\tau^b_{k-1}\le i_{\ast}<i^{\ast}\le \tau^b_k.
\] 
We further point out that
$i^{\ast}\in \{\tau^b_k, \tau^b_k-1\}$. Indeed, if $i^{\ast}<\tau^b_k$, then $\sigma$ being a GEO implies that  $\sigma(i^{\ast})$ is adjacent to $\sigma(i^{\ast}+1)$; it follows that the two vertices must have opposite colours.

Let us set  $A'_k=\{\sigma(i): i_{\ast}< i < \tau^b_k\}$ 
and show that 
\begin{equation}
\label{id: Jwk}
A_k\setminus  \bigcup_{j\le k-1}\cO^w_j  \subseteq A_k'\subseteq \cJ^w_k.
\end{equation}
Let $\tau^b_{k-1} < i\le i_{\ast}$; then as in the previous case, we can show that $\sigma(i)\in \cup_{j\le k-1}\cO^w_j$
by identifying the lowest ranking black vertex in its neighbourhood. This shows that $A_k\setminus A_k'\subseteq \cup_{j\le k-1}\cO^w_j$, which then implies the first inclusion in~\eqref{id: Jwk}. 
For the second, we note that an element of $A'_k$ is either ranked between $i_{\ast}$ and $i^{\ast}$, and therefore disconnected from the rest of the graph due to the absence of black vertices, or the white lead vertex $\sigma(i_{\ast})=\sigma(\tau^b_k-1)$. In both cases, they belong to $\cJ^w_k$.   
This also completes the proof of~\eqref{id: Kset''} in {\bf Case II}. 

The two cases combined completes the proof for~\eqref{id: Kset''}, which in turn proves~\eqref{id: Kset'} and the first identity in~\eqref{id: Kset}. For the second identity in~\eqref{id: Kset} concerning the black vertices, we can similarly reduce its proof to showing that 
\begin{equation}
\label{id: Kset-b}
\Sigma^b(\tau^b_k)\setminus\Big(\bigcup_{j\le k-1}\cO^b_j\Big) \subseteq \cI^b_k.
\end{equation}
For $k=1$, this is true as both sides are equal to $\{\sigma^b(1)\}$. For the induction step, it suffices to show that 
\begin{equation}
\label{id: Kset-b'}
\{\sigma^b(k)\}\setminus \Big(\bigcup_{j\le k-1}\cO^b_j\Big)  \subseteq  \cI^b_k. 
\end{equation}
If $\sigma^b(k)\in \cI^b_k$, there is nothing to prove. Suppose that $\sigma^b(k)$ is not a root vertex. 
We follow the same arguments of Case 2 in the proof of Proposition~\ref{prop: prim-explore}. By considering the path from $\sigma^b(k)$ to the lowest ranking black vertex in the connected component, we can find some $j\le k-1$ so that $\sigma^b(k)\in \cN^2(\sigma^b(j))$. Let  $j_{\ast}$ be the smallest such $j$. Then $\sigma^b(k)\notin \cup_{j'<j_{\ast}}\cO^b_{j'}$ and $\sigma^b_k\notin \Sigma^b(\tau^b_{j_{\ast}})$. It follows that $\sigma^b_k\in \cO^b_{j_{\ast}}\subseteq \cup_{j\le k-1}\cO^b_j$, which implies~\eqref{id: Kset-b'} and completes the proof.   

Finally, for {\it (iv)}, let us first show the first inclusion there. For $k=1$, we have 
\[
\cN\big(\sigma^b(1)\big) \subseteq \big(\cN\big(\sigma^b(1)\big) \cap \cK^w_1 \big) \cup (V^w_n\setminus \cK^w_1)  = \cO^w_1 \cup \cJ^w_1,
\]
since $\cJ^w_1=\Sigma^w(\tau^b_1)=V^w_n\setminus \cK^w_1$. 
This proves the case $k=1$. Assuming $\cN(\Sigma^b(\tau^b_{k-1}))\subseteq \cup_{j\le k-1}\cO^w_j\cup\cJ^w_{k-1}$, let us show that 
\begin{equation}
\label{eq: nsets}
\cN\big(\sigma^b(k)\big)\setminus \cO^w_k \subseteq \bigcup_{j\le k-1}\cO^w_j\bigcup \cJ^w_k,
\end{equation}
which will imply the induction step. 
However, by the definition of $\cO^w_k$, we have $\cN\big(\sigma^b(k)\big)\setminus \cO^w_k\subseteq V^w_n\setminus \cK^w_k$. The inclusion in~\eqref{eq: nsets} now follows from the first identity in~\eqref{id: Kset}. For the identity concerning $\cA^b_k$, we deduce from~\eqref{id: Abk} and~\eqref{id: AB1} that
\[
\cA^b_k=\bigcup_{j\in [k]}\cN^2(\sigma^b(j))\setminus \Sigma^b(\tau^b_k) = \bigcup_{j\in [k]}\cO^b_j \setminus \Sigma^b(\tau^b_k). 
\]
This completes the proof of {\it (iv)}.
\end{proof}

We next present a decomposition of $\cO^b_k$. For $k\in [n_b]$, 
if $\cO^w_k$ is nonempty, list its members in increasing order of their Prim's ranks as $(v_{k, i})_{1\le i\le O^w_k}$. Let $\cK_{k, 1}=\cK^b_k$; for $1\le i\le O^w_k$, set
\begin{equation}
\label{def: Ni}
\cN_i=\cN(v_{k, i})\cap \cK_{k, i}\quad\text{and}\quad  \cK_{k, i+1} = \cK_{k, i}\setminus \cN_{i}.
\end{equation}
We also set 
\[
\cN_0=\cN\big(\sigma(\tau^b_k-1)\big)\cap \cK_{k, O^w_k+1}
\]
if $\sigma(\tau^b_k-1)\in \cJ^w_k\cap \cN(\sigma(\tau^b_k))$, and $\cN_0=\varnothing$ otherwise. 

\begin{Lem}
\label{lem: decomp}
For each $k\in [n_b]$, the sets $\cN_i, 0\le i\le O^w_k$, are pairwise disjoint, and we have
\begin{equation}
\label{id: decomp}
\cO^b_k = 
\bigcup_{0\le i\le O^w_k}\cN_i.
\end{equation}
\end{Lem}

\begin{proof}
It is straightforward from the definition that $\cN_i\cap \cN_{i'}=\varnothing$ for any $0\le i'<i\le |\cO^w_k|$. 
To see why~\eqref{id: decomp} is true, we first deduce from~\eqref{def: Ni} that 
\[
\bigcup_{1\le i\le O^w_k}\cN_i=\cN(\cO^w_k)\cap \cK^b_k. 
\]
On the other hand, it is clear from its definition that $\cN_0\subseteq \cN^2(\sigma^b(k))\cap \cK^b_k=\cO^b_k$. 
Hence, for~\eqref{id: decomp} to hold, it suffices to show that 
\begin{equation}
\label{id: dec'}
\cO^b_k\setminus \cN(\cO^w_k) \subseteq \cN_0.
\end{equation}
Let $u\in \cO^b_k\setminus \cN(\cO^w_k)$, which in particular implies $u\in \cK^b_k\subseteq V^b_n\setminus \Sigma^b(\tau^b_k)$.  
Since $u\in \cN^2(\sigma^b(k))$, we can find some $v\in \cN(\sigma^b(k))$ so that $u\in \cN(v)$. According to~\eqref{eq: Nset}, we find that either $v\in \cup_{j\le k}\cO^w_j$ or $v\in\cJ^w_k$. From $u\notin\cN(\cO^w_k)$, we deduce that $v\notin \cO^w_k$. If $v\in \cO^w_j$ for some $j\le k-1$, then applying~\eqref{eq: Nset} and~\eqref{id: Kset}, we find
\[
u\in \cN^2\big(\Sigma^b(\tau^b_{k-1})\big)\subseteq \bigcup_{j\le k-1}\cO^b_j\subseteq V^b_n\setminus \cK^b_k,
\]
which contradicts with $u\in \cK^b_k$. So the only possibility is $v\in \cJ^w_k$. As $u\in \cN(v)$, Lemma~\ref{lem: Kset} {\it (i)} implies that $v=\sigma(\tau^b_j-1)\in \cJ^w_k$. However, if $j\le k-1$, then we would find $u\in \cN^2(\Sigma^b(\tau^b_{k-1}))$, which is not permitted by the previous arguments. We conclude that $v=\sigma(\tau^b_k-1)\in \cJ^w_k$ and $u\in \cN(v)$. Since $u\notin\cN(\cO^w_k)$, this then implies~\eqref{id: dec'}. 
\end{proof}

Let $n, k\in \N$ and $p\in [0, 1]$. 
Recall the notation $\cB(n, p)$ for the binomial distribution with parameters $n$ and $p$. We denote by $\cD(n, k, p)$ the distribution of the following sum of random variables: 
\begin{equation}
\label{def: D-dist}
X= \sum_{i=1}^k X_i, 
\end{equation}
where $X_1\sim \cB(n, p)$ and for $1\le i < k$, conditional on $(X_1, X_2, .., X_i)$, $X_{i+1}$ has the binomial distribution with parameters $n-\sum_{j=1}^{i} X_j$ and $p$. We also let $\cD(n, 0, p)$ denote the law of degenerate random variable $X\equiv 0$. 

Recall $K^w_k$, $K^w_k$, $O^b_k$ and $O^w_k$ from~\eqref{def: O-size}. 
The following result describes the distributions of $O^w_k$ and $O^b_k$. Recall the filtrations $(\cG^w_k)_{k\ge 1}$ and $(\cG^b_k)_{k\ge 1}$ from Section~\ref{sec: app}. Roughly speaking, $\cG^w_k$ represents the information available prior to the exploration of the $k$-th black vertex in Algorithm~\ref{algo: graph} with respect to Prim's ranking, while $\cG^b_k$ is the analogue for the 2-neighbourhood exploration in Algorithm~\ref{algo: 2-explore}. 

\begin{Prop}
\label{prop: incre}
Let $p\in [0, 1]$. The following statements hold for $G(n_b, n_w, p)$ and $k\in [n_b]$.
\begin{enumerate}[(i)]
\item 
Conditional on $\cG^w_k$, we have 
\[
O^w_{k} \sim \cB\big(K^w_{k}, p\big).
\]
\item 
Let $\cN_i$ be as in~\eqref{def: Ni} and set $\Delta_{k} = \sum_{1\le i\le O^w_k} |\cN_i|$. Then $O^b_k\ge \Delta_k$; if $O^b_k>\Delta_k$, then $\sigma(\tau^b_k-1)$ is a white lead vertex. 
Conditional on $\cG^b_k$, we have
\[
\Delta_{k}  \sim \mathcal{D}\big(K^b_k, O^w_{k}, p\big).
\] 
\end{enumerate}
\end{Prop}

\begin{proof}
Recall that $\sigma^b(k)=\sigma(\tau^b_k)$. 
For {\em (i)}, we have  
\[
\cO^w_k = \cN\big(\sigma^b(k)\big)\cap \cK^w_{k} = \big\{v\in \cK^w_{k}: U_{\{\sigma^b(k), v\}}\le p\big\}.
\]
Since $L_k:=\{e=\{\sigma^b(k), v\}: v\in \cK^w_k\}$ is contained in the set $G^w_k$ introduced in~\eqref{def: Gwkset} and is $\cG^w_k$-measurable, we apply Lemma~\ref{lem: Gk} to see that conditional on $\cG^w_k$, $\{U_e: e\in L_k\}$ are i.i.d.~uniformly distributed. It follows that conditional on $\cG^w_k$, $O^w_k$ has the $\cB(K^w_k, p)$ distribution. 

Proceeding to {\it (ii)}, we note that $O^b_k\ge \Delta_k$ is an immediate consequence of Lemma~\ref{lem: decomp}. Moreover, $O^b_k>\Delta_k$ implies $\cN_0\ne\varnothing$, in which case $\sigma(\tau^b_k-1)\in \cJ^w_k$ by the definition of $\cN_0$. 
For the distribution of $\Delta_k$,  it suffices to consider the case $O^w_k=m\in \N$, since the claim is trivially true if $O^w_k=0$. We have 
\[
\cN_1=\{u\in \cK^b_{k}: U_{\{u, v_{k, 1}\}}\le p\}.
\]
Since $v_{k, 1}\in \cO^w_k\subseteq\cK^w_k$, the edge set $\{\{u, v_{k, 1}\}: u\in \cK^b_{k}\}$ is a subset of the set $G^b_k$ defined in~\eqref{def: Gbkset}. According to Lemma~\ref{lem: Gk2} {\em (ii)}, we have $|\cN_1|\sim \cB(K^b_k, p)$ conditional on $\cG^b_k$. More generally, let us denote 
\[
M_i = \big\{\{u, v_{k, i}\}: u\in \cK^b_k\big\}, \quad 1\le i\le m.
\]
Then $M_i, 1\le i\le m$, are disjoint subsets of $G^b_k$. Lemma~\ref{lem: Gk2} {\em (ii)} applies to yield that 
\[
\{U_e: e\in M_i\}, \quad 1\le i\le m,
\]
are independent, and each is distributed as a collection of i.i.d.~uniform variables given $\cG^b_k$. 
We note that 
\[
\cN_i = \{u\in \cK_{k, i}: U_{\{u, v_{k, i}\}}\le p\} 
\]
is a function of $\{U_e: e\in M_i\}$ and $(\cN_j)_{j<i}$. 
We conclude that conditional on $\cG^b_k$ and $(\cN_j)_{j<i}$, $|\cN_i|\sim \cB(|\cK_{k, i}|, p)$. 
This then implies {\em (iii)} as $|\cK_{k, i}|=K^b_k-\sum_{j<i}|\cN_j|$. 
\end{proof}

Based on the distributions identified in Proposition~\ref{prop: incre}, we next derive upper and lower bounds for $O^w_k$ and $O^b_k$. We will need the following elementary property of the binomial distribution, whose proof is omitted.

\begin{Lem}
\label{lem: binom}
Let $\cG$ be a sigma-algebra and $m\in \N$, $p\in [0, 1]$. Suppose that $Y$ is $\Z$-valued random variable satisfying $0\le Y\le m$ and $Y\in \cG$. Suppose further that conditional on $\cG$, $X\sim \cB(Y, p)$. 
Let $\xi$ be such a random variable that conditional on $\cG$, $\xi\sim \cB(m-Y, p)$ and is independent of $X$. Then $X+\xi$ has the $\cB(m, p)$ distribution and is independent of $\cG$.
\end{Lem}

We recall that $\cJ^w_k$ contains all the white lead vertices of ranks up to $\tau^b_k$. We denote by $J^w(k)$ its cardinality. 

\begin{Lem}[Upper bounds for $O^w_k$ and $O^b_k$]
\label{lem: upbd}
Let $p\in [0, 1]$ and $\Delta_k$ be as in Proposition~\ref{prop: incre}. 
There exist sequences of random variables $(\bar\xi^w_k)_{k\in [n_b]}$ and $(\bar\xi^b_k)_{k\in [n_b]}$ such that 
\begin{enumerate}[(i)]
\item 
$\{\bar\xi^w_k: k\in [n_b]\}$ is a collection of i.i.d.~with common distribution $\cB(n_w, p)$;
\item
conditional on $(O^w_k)_{k\in [n_b]}$, $\{\bar\xi^b_k: k\in [n_b]\}$ is a collection of independent random variables where $\bar\xi^b_k$ is distributed according to $\cB(n_b\cdot O^w_k, p)$;
\item 
$O^w_k\le \bar\xi^w_k$ and $\Delta_k\le \bar\xi^b_k$ for all $k\in [n_b]$.
\end{enumerate}
Moreover, we can find a sequence of i.i.d.~random variables $(\chi_{i})_{i\ge 1}$ with common distribution $\cB(n_b, p)$ such that for all $k_0\in [n_b]$, 
\begin{equation}
\label{bd: Delta}
\sum_{1\le k\le k_0}(O^b_k-\Delta_k) \le \sum_{1\le i\le J^w(k_0)} \chi_{i}\,;
\end{equation}
\end{Lem}

\begin{proof}
We begin with the construction of $(\bar \xi^w_k)_{k\in [n_b]}$. For $k\in [n_b]$, let $\xi_k$ be a random variable such that conditional on $\cG^w_k$, it is distributed according to $\cB(n_w-K^w_k, p)$ and independent of $\mathcal U$ and $(\xi_i)_{i\le k-1}$; also let $\bar\cG^w_k$ be the sigma-algebra generated by $\cG^w_k$ and $(\xi_i)_{i\le k-1}$. We then set for each $k\in [n_b]$: 
\[
\bar \xi^w_k = O^w_k + \xi_k\ge O^w_k. 
\]
Let us argue that $\bar\xi^w_k\in \bar\cG^w_{k+1}$, $\bar\xi^w_k\sim \cB(n_w, p)$ and is independent of $\bar\cG^w_{k}$, which will imply {\it (i)}. 
The fact that $\bar\xi^w_k\in \bar\cG^w_{k+1}$ is a direct consequence of $O^w_k\in \cG^w_{k+1}$ (Lemma~\ref{lem: Gk}) and the definition of $\bar\cG^w_k$. We will prove the two remaining properties by induction. 
For $k=1$, this is a direct consequence of Lemma~\ref{lem: binom} since $\bar\cG^w_1=\cG^w_1$. Assume that the properties hold for all $i\le k-1$.  Proposition~\ref{prop: incre} says that conditional on $\cG^w_k$, $O^w_k\sim\cB(K^w_k, p)$. Since by construction $\xi_i$ is conditionally independent of $\mathcal U$ given $\cG^w_i$, and since $\cG^w_k$ is obtained from $\cG^w_i$ and a subset of $\mathcal U$, $\xi_i$ is also conditionally independent of $\mathcal U$ given $\cG^w_k$, and therefore conditionally independent of $O^w_k$ given $\cG^w_k$. As $\bar\cG^w_k$ is obtained from $\cG^w_k$ and $(\xi_i)_{i\le k-1}$, we deduce that given $\bar\cG^w_k$, $O^w_k$ is still distributed as $\cB(K^w_k, p)$. Applying Lemma~\ref{lem: binom} yields that $\bar\xi^w_k\sim\cB(n_w, p)$ and is independent of $\bar\cG^w_k$. 

The construction for $(\bar\xi^b_k)_{k\in [n_b]}$ is similar albeit slightly more involved. Recall the sets $\cN_i$ from~\eqref{def: Ni}. Proposition~\ref{prop: incre} combined with~\eqref{def: D-dist} says that $\Delta_k=\sum_{1\le i\le O^w_k}|\cN_i|$ is a sum of binomial variables, with $|\cN_1|\sim \cB(K^b_k, p)$. 
Note that $O^w_k\in \cG^b_k$. 
Applying Lemma~\ref{lem: binom} to $|\cN_1|$ and $\cG^b_k$, we can find a random variable $\xi_{k, 1}$ such that given $\cG^b_k$, $\xi_{k, 1}\sim \cB(n_b, p)$ and $|\cN_1|\le \xi_{k, 1}$. We then note that conditional on $\cG^b_k$ and $|\cN_1|$, $|\cN_2|\sim \cB(K^b_k-|\cN_1|, p)$. Applying Lemma~\ref{lem: binom} allows us to find some $\xi_{k, 2}$ such that given $\cG^b_k$ and $|\cN_1|$, $\xi_{k, 2}\sim \cB(n_b, p)$, is independent of $\xi_{k, 1}$ and satisfies $|\cN_2|\le \xi_{k, 2}$. Iterating this procedure, we can then find a collection $\{\xi_{k,i}: 1\le i\le O^w_k, k\in [n_b]\}$ such that for $k\in [n_b]$, 
conditional on $\cG^b_{k}$ and $\{\xi_{j, i}: 1\le i\le O^w_i, j\le k-1\}$, $\{\xi_{k, i}: 1\le i\le O^w_k\}$ is 
i.i.d. $\cB(n_b, p)$-distributed, independent of $\mathcal U$ and verifies
\[
\Delta_k\le \bar\xi^b_k:=\sum_{i=1}^{O^w_k}\xi_{k, i}. 
\]
It follows that given $(O^w_i)_{i\le k}$, $(\bar\xi^b_i)_{i\le k}$ is an independent sequence and $\bar\xi^b_i\sim \cB(O^w_i\cdot n_b, p)$. 

It remains to construct $(\chi_{\ell})_{\ell\ge 1}$. 
Recall $\zeta_{\ell}$ and $\cH_{\ell}$ from Section~\ref{sec: app}. 
Lemma~\ref{lem: Hk} implies that $\{\sigma(\zeta_{\ell})\in V^w_n\}$ is $\cH_{\ell}$-measurable. Moreover, since $\sigma(\zeta_{\ell})$ is a lead vertex, and therefore disconnected from any lower ranking vertices, we obtain from Lemma~\ref{lem: Hk} that given $\cH_{\ell}$ and $\sigma(\zeta_{\ell})\in V^w_n$, we have 
\[
\big|\cN\big(\sigma(\zeta_{\ell})\big)\big| = \big|\cN\big(\sigma(\zeta_{\ell})\big)\setminus \Sigma^b(\zeta_{\ell})\big|\sim \cB(n_b-|\Sigma^b(\zeta_{\ell})|, p).
\]
Applying Lemma~\ref{lem: binom} then allows us to find  a sequence $(\chi'_{\ell})_{\ell\ge 1}$ such that given $\cH_{\ell}$, $\chi'_{\ell}\sim \cB(n_b, p)$ and is independent of $(\chi'_i)_{i<\ell}$, and  for all $\ell\ge 1$, we also have
\begin{equation}
\label{bd: zetal}
\big|\cN\big(\sigma(\zeta_{\ell})\big)\big|  \mathbf 1_{\{\sigma(\zeta_{\ell})\in V^w_n\}}\le \chi'_{\ell}\mathbf 1_{\{\sigma(\zeta_{\ell})\in V^w_n\}}. 
\end{equation}
It follows that $\{\chi_{i}: i\ge 1\}:=\{\chi'_{\ell}: \sigma(\zeta_{\ell})\in V^w_n, \ell\ge 1\}$ is i.i.d.~with common distribution $\cB(n_b, p)$. 

It remains to show~\eqref{bd: Delta}. To that end, we recall from Lemma~\ref{lem: decomp} and Proposition~\ref{prop: incre} that $O^b_k-\Delta_k=|\cN_0|$. We also recall from the definition of $\cN_0$ that $|\cN_0|>0$ implies that we can find some $\ell\ge 1$ so that $\tau^b_k-1=\zeta_{\ell}$ and $\sigma(\zeta_{\ell})\in V^w_n$. Hence, 
\begin{align*}
O^b_k-\Delta_k & =  \sum_{\ell\ge 1}|\cN_0|\mathbf 1_{\{\tau^b_k-1=\zeta_{\ell}, \sigma(\zeta_{\ell})\in V^w_n\}} 
\le \sum_{\ell\ge 1}\big|\cN(\sigma(\zeta_{\ell}))\big|\mathbf 1_{\{\tau^b_k-1=\zeta_{\ell},\sigma(\zeta_{\ell})\in V^w_n\}} \\ 
&\le  \sum_{\ell\ge 1}\chi'_{\ell}\mathbf 1_{\{\tau^b_k-1=\zeta_{\ell}, \sigma(\zeta_{\ell})\in V^w_n\}},
\end{align*}
where in the second inequality above we have used the definition of $\cN_0$ and~\eqref{bd: zetal} in the last line. 
Summing over $k\le k_0$ we find that
\[
\sum_{k\le k_0}(O^b_k-\Delta_k)\le \sum_{k\le k_0}\sum_{\ell\ge 1}\chi'_{\ell}\mathbf 1_{\{\tau^b_k-1=\zeta_{\ell}, \sigma(\zeta_{\ell})\in V^w_n\}} \le  \sum_{\ell\ge 1}\chi'_{\ell}\mathbf 1_{\{\zeta_{\ell}\le \tau^b_{k_0}, \,\sigma(\zeta_{\ell})\in V^w_n\}}=\sum_{\ell\ge 1}\chi'_{\ell}\mathbf 1_{\{\sigma(\zeta_{\ell})\in \cJ^w_{k_0}\}}, 
\]
where the last term is equal to the sum of the first $J^w(k_0)$ entries of $(\chi_i)_{i\ge 1}$. The proof is complete.
\end{proof}

\begin{Lem}[Lower bounds for $O^w_k$ and $O^b_k$]
\label{lem: lowbd}
Let $\underline n_w, \underline n_b\in \N$. There exist sequences of random variables $(\underline \xi^w_k)_{k\in [n_b]}$ and $(\underline \xi^b_k)_{k\in [n_b]}$ that satisfy 
\begin{enumerate}[(i)]
\item 
$\{\underline \xi^w_k: k\in [n_b]\}$ is a collection of i.i.d.~random variables with common distribution $\cB(\underline n_w, p)$;
\item
conditional on $(O^w_k)_{k\in [n_b]}$, $\{\underline \xi^b_k: k\in [n_b]\}$ is a collection of independent variables and $\underline\xi^b_k\sim \cB(O^w_k \underline n_b, p)$ for each $k\in [n_b]$.  
\item 
$O^w_k\ge \underline \xi^w_k\mathbf 1_{\{K^w_k \ge \underline n_w\}}$ and $O^b_k\ge \underline \xi^b_k\mathbf 1_{\{K^w_{k+1}\ge \underline n_b\}}$ for all $k\in [n_b]$.
\end{enumerate}
\end{Lem}

\begin{proof}
Let us denote the event $A_k=\{K^w_k\ge \underline n_w\}$. We note that $A_k\in \cG^w_k$ thanks to Lemma~\ref{lem: Gk}. 
Moreover, we have seen in the proof of Proposition~\ref{prop: incre} that given $\cG^w_1$, 
\[
\big\{\mathbf 1_{U_e\le p}: e=\{\sigma^b(1), v\}, v\in \cK^w_1\}
\]
is a collection of i.i.d.~Bernoulli$(p)$ variables. On the event $A_1$, let $\cK'_1$ be the subset of $\cK^w_1$ consisting of the $\underline n_w$ lowest ranking elements, and note that $\cK'_1\in\cG^w_1$. We then set 
\[
\underline \xi^w_1 = \big|\big\{\mathbf 1_{U_e\le p}: e=\{\sigma^b(1), v\}, v\in \cK'_1\big\}\big|
\]
on $A_1$, while on $A_1^c$, simply let $\underline\xi^w_1$ be a variable of $\cB(\underline n_w, p)$-distribution, independent of $\cG^w_1$.
It is then clear that $\underline\xi^w_1 \le O^w_1$ on $A_1$ and $\underline\xi^w_1 \sim \cB(\underline n_w, p)$. Since the distribution of  $\underline\xi^w_1$ remains the same on $A_1$ and $A_1^c$, we deduce that $\underline\xi^w_1$ is in fact  
independent of $\cG^w_1$. Let $\underline\cG^w_2$ be the sigma-algebra generated by $\cG^w_2$ and $\underline\xi^w_1$. Equivalently, $\underline\cG^w_2$ is generated by $\cG^w_2$ and $\underline\xi^w_1\mathbf 1_{A^c_1}$. Due to the independence of $\underline\xi^w_1\mathbf 1_{A^c_1}$ from $\cG^w_1$, the distribution of $O^w_2$ given $\underline\cG^w_2$ is identical to its distribution given $\cG^w_2$. We can then repeat the previous procedure to obtain a sequence $(\underline\xi^w_k)_{k\in [n_b]}$ and the filtration $\underline\cG^w_k$ generated by $\cG^w_k$ and $(\underline\xi^w_j\mathbf 1_{A^c_j})_{j\le k-1}$ so that for $1\le k\le n_b$, 
\[
\underline\xi^w_k\sim\cB(\underline n_w, p), \quad \underline\xi^w_k\mathbf 1_{A_k}\le O^w_k, \quad \underline\xi^w_k\in \underline\cG^w_{k+1}, \ \text{ and } \underline\xi^w_k \text{ is independent of } \underline\cG^w_{k}. 
\]
This completes the construction of $(\underline\xi^w_k)_{k\ge 1}$. 

The construction of $\underline \xi^b_k$ is similar. Since $O^b_k\ge \Delta_k$, as seen in Proposition~\ref{prop: incre}, we only need to provide a lower bound for $\Delta_k=\sum_{1\le i\le O^w_k}|\cN_i|$, where $\cN_i$ has the $\cB(|\cK_{k, i}|, p)$-distribution. 
On the event $|\cK_{k, i}|\ge \underline n_b$, as in the previous case we can find some $\xi_{k, i}\sim \cB(\underline n_b, p)$ satisfying $\xi_{k, i}\le |\cN_i|$. On the event  $|\cK_{k, i}|< \underline n_b$, simply let $\xi_{k, i}\sim \cB(\underline n_b, p)$ be an independent variable. Note that  $|\cK_{k, i}|< \underline n_b$ implies $\{K^b_{k+1}< \underline n_b\}$. It follows that 
\[
O^b_k \ge \underline\xi^b_k\mathbf 1_{K^b_{k+1}\ge \underline n_b}, \quad\text{where } \underline\xi^b_k:=\sum_{1\le i\le O^w_k}\xi_{k, i}. 
\]
The distribution properties of $(\underline\xi^b_k)$ can be argued similarly as in the previous case.  
\end{proof}

We end the section with a technical note on the previous proofs. In the proof of Lemma~\ref{lem: upbd}, it might appear more natural to take $\hat\xi_k:=|\{\mathbf 1_{U_e\le p}: e=\{\sigma^b(k), u\}, u\in V^w_n\}|$ as the upper bound. The issue with it is that $\sigma^b(k)$ is a function of the weights of the edges adjacent to $\Sigma(\tau^b_k-1)$, and consequently $\hat\xi_k$ is dependent of $(\hat\xi_i)_{i<k}$. By contrast, our choice of the lower bound in proving Lemma~\ref{lem: lowbd} does not have this problem, as we restrict to the edges between $\sigma^b(k)$ and $\cK^w_k\subseteq V^w_n \setminus \Sigma^w(\tau^b_k)$.

\section{Proof in the sublinear regime } 
\label{sec: sublinear}

We prove Theorem~\ref{thm: sublinear} in this section. In fact, we will prove the following slightly stronger result. Recall from Theorem~\ref{thm: sublinear} that $\gamma_{\theta}=\sqrt{\frac{1-\theta}{\theta}}$.

\begin{The}
\label{thm: sublinear'}
Assume that~\eqref{hyp: theta} holds. 
Let $(\kappa_n)_{n\ge 1}$ be a sequence of positive integers such that $\kappa_n\to\infty$ and $\kappa_n/n\to 0$ as $n\to\infty$. Then we have
\[
\sup_{0\le t\le 1}\bigg|\frac{1}{\kappa_n}\big|\Sigma^b(\lfloor \kappa_n t\rfloor)\big|-\frac{t}{1+\gamma_{\theta}}\bigg|\xrightarrow{n\to\infty}0 \quad\text{in probability.}
\]
\end{The}

We note that Theorem~\ref{thm: sublinear} is an immediate consequence of Theorem~\ref{thm: sublinear'} by taking $t=1$. 

The rest of this section is organised as follows. In Section~\ref{sec: sub-overview} we lay out the main steps in the proof of Theorem~\ref{thm: sublinear'}. We defer some of the more technical proofs to Sections~\ref{sec: boundJ} and~\ref{sec: limit-sub}.

\subsection{An overview of the proof in the sublinear regime}
\label{sec: sub-overview}

We recall the random graph $G(n_b, n_w, p)$ with the edge set $E_n(p)$ introduced in~\eqref{def: Enp}. Throughout Section~\ref{sec: sublinear}, we take $p$ to be critical, i.e. 
\[
p=\frac{1}{\sqrt{n_bn_w}}.
\]
Recall that $\tau^b_k$ is the Prim rank of the $k$th black vertex. 
The proof of Theorem~\ref{thm: sublinear'} consists in estimating $|\Sigma^b(\lfloor \kappa_n t\rfloor)|$, the number of black vertices among the first $\lfloor \kappa_n t\rfloor$ entries in the Prim sequence. 
As a first step towards this end, we consider the set $\Sigma^w(\tau^b_k)$ and provide an approximation of its size using quantities that appear in the exploration of $G(n_b, n_w, p)$. More specifically, 
recall that $\cJ^w_k$ is the set of white lead vertices of ranks up to $\tau^b_k$ and $J^w(k)=|\cJ^w_k|$; 
for $k\in [n_b]$, let
\begin{equation}
\label{id: Sk}
R(k)=|\Sigma^w(\tau^b_k)|-J^w(k).
\end{equation}

\begin{Prop}
    \label{prop: bound-compo}
    Let the assumptions of Theorem~\ref{thm: sublinear'} hold and let $p=\frac{1}{\sqrt{n_bn_w}}$. 
The following convergence holds in probability: 
    \[
\frac{1}{\kappa_n}J^w\big(|\Sigma^b(\kappa_n)|\big) \xrightarrow{n\to\infty} 0. 
    \]
\end{Prop}

\begin{Prop}
\label{prop: LLN}
Let the assumptions of Theorem~\ref{thm: sublinear'} hold and let $p=\frac{1}{\sqrt{n_bn_w}}$. 
The following convergence holds in probability: 
\[
\frac{1}{\kappa_n}\sup_{j \le|\Sigma^b(\kappa_n)|}\big|R(j) - \gamma_\theta j\big| \xrightarrow{n\to\infty} 0.
\]
\end{Prop}

Proposition~\ref{prop: bound-compo} is proved in Section~\ref{sec: boundJ} and Proposition~\ref{prop: LLN} in Section~\ref{sec: limit-sub}. Both proofs rely on the 2-neigbhourhood exploration of $G(n_b, n_w, p)$ in Algorithm~\ref{algo: 2-explore}. 

\paragraph{The dual problem.} At various points in our proof, we will employ the following coupling. 
By flipping the vertex colours--black to white and white to black--we obtain a graph $\hat K_{n_b, n_w}$ on the same vertex set $V_n$ and the edge set $E_n$ but with reversed vertex colours. 
In particular, the proportion of black vertices in $\hat K_{n_b, n_w}$ is given by 
\[
\frac{n_w}{n} \xrightarrow{n\to\infty} \hat\theta:=1-\theta.
\]
Run Algorithm~\ref{algo: Prim} on $\hat K_{n_b, n_w}$ with the same edge weights $\{U_{e}: e\in E_n\}$ and in {\bf Step 1} start with the same $\sigma(1)$, which is also uniformly distributed on the vertex set of $\hat K_{n_b, n_w}$. 
Then it is not difficult to see that the algorithm will produce the same Prim sequence $(\sigma(k))_{k\in[n]}$ as well as the same Prim edges $(e_k)_{k\in [n-1]}$. Let $\tau^w_k$ be the rank of the $k$th white vertex in $(\sigma(i))_{i\in [n]}$ for $K_{n_b, n_w}$. 
Define $\hat\Sigma^b(k), \hat\Sigma^w(k)$, $\hat\tau^b_k$, $\hat\tau^w_k$ to be the analogues of $\Sigma^b(k)$,  $\Sigma^w(k)$, $\tau^b_k$, $\tau^w_k$ for $\hat K_{n_b, n_w}$. Note that we have
\begin{equation}
\label{id: col-swap}
\hat\Sigma^w(k) = \Sigma^b(k), \quad \hat\Sigma^b(k)=\Sigma^w(k), \quad \hat\tau^w_k = \tau^b_k. 
\end{equation}
Let $\hat G(n_b, n_w, p)$ be the subgraph of $\hat K_{n_b, n_w}$ with the edge set $E_n(p)$. Then $\hat G(n_b, n_w, p)$ has the same distribution as $G(n_w, n_b, p)$; in particular, all our previous estimates apply to $\hat G(n_b, n_w, p)$. 

\medskip
We are now ready to prove Theorem~\ref{thm: sublinear'}. 
\begin{proof}[Proof of Theorem~\ref{thm: sublinear'}]
Let us first show that
\begin{equation}
\label{cv: Sigma}
\rho^{(n)}_{\kappa_n}=\frac{|\Sigma^b(\kappa_n)|}{\kappa_n}\xrightarrow{n\to\infty} \frac{1}{1+\gamma_{\theta}} \quad \text{in probability.}
\end{equation}
To ease the notation, let us denote $\eta_n=|\Sigma^b(\kappa_n)|$, which is bounded by both $ \kappa_n$ and $n_b$. 
Applying in tandem~\eqref{id: Sk}, Proposition~\ref{prop: bound-compo} and Proposition~\ref{prop: LLN}, we find that
\[
\big|\Sigma^w\big(\tau_{\eta_n}^b\big)\big| = R(\eta_n) + J^w(\eta_n) = R(\eta_n)+ o_p(\kappa_n) = \gamma_\theta \eta_n  + o_p({\kappa_n}), \quad n\to\infty.
\]
Since by definition,  $|\Sigma^b(\tau_j^b)| = j$ for each $j\le n_b$, we can re-write the previous identity as follows:
\begin{equation}
\label{eq: Sigmab}
\big|\Sigma^w\big(\tau^b_{\eta_n}\big)\big| = \gamma_\theta \big|\Sigma^b\big(\tau_{\eta_n}^b\big)\big| + o_p({\kappa_n}), \quad n\to\infty.
\end{equation}
Recall $\hat G(n_b, n_w, p)$ from above, where the proportion of black vertices converges to $1-\theta$. As noted previously, all our previous estimates also apply to $\hat G(n_b, n_w, p)$, so that we have the following analogue of~\eqref{eq: Sigmab}:
 \[
\big|\hat\Sigma^w\big(\hat\tau^b_{|\hat\Sigma^b(\kappa_n)|}\big)\big| = \gamma_{1-\theta} \big|\hat\Sigma^b\big(\hat\tau_{|\hat\Sigma^b(\kappa_n)|}^b\big)\big| + o_p({\kappa_n}), \quad n\to\infty.
 \]
By the identity~\eqref{id: col-swap} and the fact that $\gamma_{1-\theta}=\gamma_{\theta}^{-1}$, we find that
\begin{equation}
\label{eq: Sigmaw}
\big|\Sigma^b\big(\tau_{|\Sigma^w(\kappa_n)|}^w\big)\big| = \gamma_\theta^{-1} \big|\Sigma^w\big(\tau_{|\Sigma^w(\kappa_n)|}^w)\big| + o_p({\kappa_n}), \quad n\to\infty. 
\end{equation}
We note that if $\sigma(\kappa_n)$ is black, then $\tau_{|\Sigma^b(\kappa_n)|}^b=\kappa_n$; otherwise $\tau_{|\Sigma^w(\kappa_n)|}^w=\kappa_n$. 
Hence,~\eqref{eq: Sigmab} and~\eqref{eq: Sigmaw} together show that
\[
\big|\Sigma^b(\kappa_n)\big| = \gamma_\theta^{-1} \big|\Sigma^w(\kappa_n)\big| + o_p(\kappa_n)=\gamma_\theta^{-1} \big(\kappa_n-\big|\Sigma^b(\kappa_n)\big|\big) + o_p(\kappa_n), \quad n\to\infty,
\]
since $\big|\Sigma^b(\kappa_n)\big| + \big|\Sigma^w(\kappa_n)\big| = \kappa_n$. Re-arranging the terms and dividing both sides by $\kappa_n$ yields~\eqref{cv: Sigma}. Now let $t\in (0, 1]$. Since the sequence $(\lfloor\kappa_nt\rfloor)_{n\in \N}$ also verifies the assumption of Theorem~\ref{thm: sublinear'},~\eqref{cv: Sigma} applies to yield that
\[
\frac{1}{\kappa_n}\big|\Sigma^b(\lfloor\kappa_nt\rfloor)\big|\xrightarrow{n\to\infty}\frac{t}{1+\gamma_\theta}\quad\text{in probability.} 
\]
We also point out that the previous convergence also holds trivially for $t=0$. Finally, the monotonicity of $|\Sigma^b(\lfloor \kappa_n t\rfloor)|$ in $t$ implies that for any $t\in [t_1, t_2]\subseteq [0, 1]$, 
\[
\frac{1}{\kappa_n}\big|\Sigma^b(\lfloor\kappa_nt_1\rfloor)\big|-\frac{t_2}{1+\gamma_{\theta}}\le \frac{1}{\kappa_n}\big|\Sigma^b(\lfloor\kappa_nt\rfloor)\big|-\frac{t}{1+\gamma_{\theta}}\le \frac{1}{\kappa_n}\big|\Sigma^b(\lfloor\kappa_nt_2\rfloor)\big|-\frac{t_1}{1+\gamma_{\theta}}.
\]
Hence, by dividing $[0, 1]$ into sufficiently small sub-intervals, we can deduce the claimed uniform convergence from the previous pointwise convergence. 
\end{proof}

\subsection{Proof of Proposition~\ref{prop: bound-compo}}
\label{sec: boundJ}

Recall that  $\cJ^w_k$ and $\cI^b_k$ stand respectively for the set of the white lead vertices and root vertices from $\Sigma(\tau^b_k)$, and $J^w(k)=|\cJ^w_k|$, $I^b(k)=|\cI^b_k|$. We prove Proposition~\ref{prop: bound-compo} here, by first deriving a bound for $I^b(k)$ and then deducing a similar bound for $J^w(k)$ using the duality argument. Throughout this subsection, we work under the assumptions of Theorem~\ref{thm: sublinear'} and $p=\frac{1}{\sqrt{n_bn_w}}$. 

Recall $\cO^b_k$ from~\eqref{def: Obk} and $\cO^w_k$ from~\eqref{def: Ow}, and note that $O^w_k=|\cO^w_k|$, $O^b_k=|\cO^b_k|$. We set $S^w_0=S^b_0=0$ and for $k\in [n_b]$, 
\begin{equation}
\label{def: Swalk}
S^w_k=\sum_{j\in [k]}O^w_j, \qquad S^b_k = \sum_{j\in [k]}O^b_j,
\end{equation}
which corresponds to the respective sizes of the disjoint unions $\cup_{j\le k}\cO^w_j$ and $\cup_{j\le k}\cO^b_j$. On the other hand, recall from~\eqref{eq: Nset} that
\begin{equation}
\label{def: Ab}
\cA^b_k=\bigcup_{j\in[k]}\cO^b_j\setminus\Sigma^b(\tau^b_k)
=\Big\{\sigma^b(i): i>k, \sigma^b(i)\in\bigcup_{j\in[k]}\cO^b_j\Big\}. 
\end{equation}
Let us also denote
\begin{equation}
\label{def: Aw}
\cA^w_k =\Big\{\sigma(i): i>\tau^b_k, \,\sigma(i)\in \bigcup_{j\in [k-1]}\cO^w_j\Big\},
\end{equation}
and $A^w(k)=|\cA^w_k|$, $A^b(k)=|\cA^b_k|$. Recall $R(k)$ from~\eqref{id: Sk}. We first establish some combinatorial properties of the graph exploration process. 

\begin{Lem}
\label{lem: Aset}
For all $k\in [n_b]$, we have 
\begin{equation}
\label{id: Aset}
A^w(k)= S^w_{k-1}-R(k)\quad \text{and}\quad A^b(k) = S^b_k-k+I^b(k); 
\end{equation}
moreover, 
\begin{equation}
\label{id: I}
I^b(k)-1= -\inf_{j\in [k-1]}(S^b_j-j). 
\end{equation}
\end{Lem}

\begin{proof}
For the first identity in~\eqref{id: Aset}, we note that by definition, $\Sigma^w(\tau^b_k)\subseteq V^w_n\setminus \cK^w_k$. Using the decomposition in~\eqref{id: Kset}, we find that 
\[
\Sigma^w(\tau^b_k)\setminus \cJ^w_k \subseteq \bigcup_{j\in [k-1]}\cO^w_j.
\]
Since $\cJ^w_k$ is disjoint from any $\cO^w_j$, 
the difference of the two sets is given by
\[
\bigcup_{j\in [k-1]}\cO^w_j\setminus\Big(\Sigma^w(\tau^b_k)\setminus \cJ^w_k\Big) = \bigcup_{j\in [k-1]}\cO^w_j\setminus \Sigma^w(\tau^b_k)=\cA^w_k.
\]
On the one hand, $R(k)=|\Sigma^w(\tau^b_k)|-J^w(k)=|\Sigma^w(\tau^b_k)\setminus\cJ^w_k|$; on the other hand, $S^w_{k-1}$ is the size of the disjoint union $\cO^w_j, j\le k-1$. Combining this with the previous argument, we have shown $S^w_{k-1}-R(k)=A^w(k)$ for each $k\in [n_b]$. 

Using the second decomposition in~\eqref{id: Kset} and the fact that $\Sigma^b(\tau^b_k)\subseteq V^b_n\setminus \cK^b_k$, we find that 
\[
\Sigma^b(\tau^b_k)\setminus\cI^b_k\subseteq \bigcup_{j\le k-1}\cO^b_j\subseteq \bigcup_{j\le k}\cO^b_j. 
\]
It follows that
\[
A^b(k)=\Big|\cup_{j\le k}\cO^b_j\Big|- \Big|\Sigma^b(\tau^b_k)\setminus\cI^b_k\Big| = S^b_k - (k-I^b(k)),
\]
since $|\Sigma^b(\tau^b_k)|=k$. This completes the proof of~\eqref{id: Aset}.

To prove~\eqref{id: I}, let us first show that for $1\le k\le n_b$, 
\begin{equation}
\label{id-pf: I2}
\sigma^b(k)\in \cI^b_k \quad \Longleftrightarrow\quad  \cA^b_{k-1}=\varnothing \quad \Longleftrightarrow\quad A^b(k-1)=0,
\end{equation}
with the convention $\cA^b_0=\varnothing$. 
Indeed, if $\cA^b_{k-1}\ne\varnothing$, then  $\sigma^b(k)\in \cA^b_{k-1}=\cN^2(\Sigma^b(\tau^b_{k-1}))$. It follows that we can find some $j\le k-1$ so that $\sigma^b(k)\in \cN^2(\sigma^b(j))$. As a result, $\sigma^b(k)$ can not be the lowest ranking black vertex in its component. 
If, on the other hand, $\cA^b_{k-1}=\cN^2(\Sigma^b(\tau^b_{k-1}))=\varnothing$, then $\sigma^b(k)$ is disconnected from any member of $\Sigma^b(\tau^b_{k-1})$, and consequently $\sigma^b(k)$ is a root vertex. This proves~\eqref{id-pf: I2}. 

Denote $M(k)=-\inf_{j\le k}(S^b_j-j)$, and note that $M(k)$ is non-decreasing and has maximal jump size of 1. Write $B_k=\{ j\in [k-1]: M(j)=M(j-1)+1\}$; we have $M(k-1)=|B_k|$. On the other hand,~\eqref{id-pf: I2} implies that $I^b(k)=|B'_k|$, where $B'_k:=\{0\le j\le k-1: A^b(j)=0\}$. Noting that $0\in B'_k$, the conclusion will follow once we show that $B_k=B'_k\setminus\{0\}$. 

To that end, we assume that the elements of $B'_k\setminus\{0\}$ are $j_1<j_2<\cdots<j_{\ell}$. Then for all $1\le j<j_1$, we have $I^b(j)=1$ and $A^b(j)\ge 1$. We then deduce from the second identity in~\eqref{id: Aset} that $S^b(j)-j=A^b(j)-I^b(j)\ge 0$, so that $M(j)=0$ for all $j<j_1$. We also find $S(j_1)-j_1=A^b(j_1)-I^b(j_1)=-1$. This shows $j_1$ is the smallest element of $B_k$. Similarly, for $j_1<j<j_2$, we have $I^b(j)=2$ and $S^b(j)-j\ge -1$, as well as $S^b(j_2)=-2$. Iterating this procedure, we see that the elements of $B_k$ are given by $j_1, j_2, \dots, j_{\ell}$, which completes the proof of~\eqref{id: I} by the previous arguments.  
\end{proof}

\begin{Prop}[Law of large numbers for $S^w$]
\label{prop: LLN-w}
For all $\vep>0$, the following convergence holds:
\begin{equation}
\label{cv: Swupp}
\mathbb P\Big(\sup_{k\le |\Sigma^b(\kappa_n)|}\Big|S^w_k-\gamma_{\theta}k\Big| \ge \vep \kappa_n\Big) \xrightarrow{n\to\infty} 0.
\end{equation}
\end{Prop}

\begin{proof}
Thanks to Lemma~\ref{lem: upbd}, we have 
\[
\forall\,k\in [n_b]: \quad S^w_k\le \bar S^w_k:=\sum_{i\le k}\bar\xi^w_i.
\]
Denote $\gamma^{(n)}_{\theta}=\sqrt{n_w/n_b}$  and recall $p=1/\sqrt{n_bn_w}$. We note that Lemma~\ref{lem: upbd} also says that $(\bar\xi^w_i)_{i\in [n_b]}$ is i.i.d.~and satisfies 
\[
\mathbb E\big[\bar \xi^w_1\big]=n_wp=\gamma^{(n)}_{\theta}, \quad \Var\big(\bar \xi^w_1\big) = n_w p (1-p)\le  \gamma^{(n)}_{\theta}. 
\]
By independence, we then have $\Var(\bar S^w_k)\le k\gamma^{(n)}_{\theta}$. Doob's maximal inequality then asserts that for any $\vep>0$, 
\[
\mathbb P\Big(\sup_{k\le \kappa_n}|\bar S^w_k- \gamma^{(n)}_{\theta}k|\ge \vep\kappa_n \Big)\le \frac{4 \Var(\bar S^w_{\kappa_n})}{\vep^2\kappa_n^2}\le \frac{4\gamma^{(n)}_{\theta}}{\vep^2\kappa_n}\xrightarrow{n\to\infty} 0,
\]
since $\gamma^{(n)}_{\theta}\to \gamma_{\theta}$ under the assumption~\eqref{hyp: theta}. 
It follows that for sufficiently large $n$, 
\[
\mathbb P\Big(\sup_{k\le \kappa_n}\big|\bar S^w_k-\gamma_{\theta}k\big|\ge \vep\kappa_n\Big) \le \mathbb P\Big(\sup_{k\le \kappa_n}|\bar S^w_k- \gamma^{(n)}_{\theta}k|\ge \tfrac12\vep\kappa_n \Big) \xrightarrow{n\to\infty}0. 
\]
Since $S^w_k\le \bar S^w_k$, we deduce that
\begin{equation}
\label{bd: Sw1}
\mathbb P\Big(\sup_{k\le \kappa_n}\big(S^w_k-\gamma_{\theta}k\big)\ge \vep\kappa_n\Big)\xrightarrow{n\to\infty}0. 
\end{equation}
We will shorthand $\eta_n=|\Sigma^b(\kappa_n)|$. 
To derive a matching lower bound, we first note that thanks to Lemma~\ref{lem: Kset} {\it (ii)} and {\it (iii)}, we have 
\begin{equation}
\label{id: ksize}
n_w-K^w_{\eta_n}\le  S^w_{\eta_n} + J^w(\eta_n). 
\end{equation}
Note further that $\tau^b_{\eta_n}=\tau^b_{|\Sigma^b(\kappa_n)|}\le \kappa_n$, so that $J^w(\eta_n)\le |\Sigma(\kappa_n)|=\kappa_n$. 
Together with~\eqref{bd: Sw1} and~\eqref{id: ksize}, and the fact that $\eta_n \leq \kappa_n$ this implies $n_w-K^w_{\eta_n}=O_p(\kappa_n) = o_p(n_w)$ as $n\to\infty$. Let us take $\delta_n=\sqrt{\kappa_n/n}$, so that $\delta_n\to 0$ and 
\begin{equation}
\label{bd: gevent}
\liminf_{n\to\infty}\mathbb P\big(K^w_{\eta_n}\ge (1-\delta_n)n_w\big) = 1. 
\end{equation}
Let $Q_n$ be the event that $K^w_{\eta_n}\ge (1-\delta_n)n_w$. Applying Lemma~\ref{lem: lowbd} with $\underline n_w=(1-\delta_n)n_w$, we can find a sequence of i.i.d.~random variables $(\underline \xi^w_k)_{k\in [n_b]}$ with common distribution $\cB((1-\delta_n)n_w, p)$ such that
\begin{equation}
\label{bd: Sw2}
\forall\,k\in [n_b]:\quad S^w_k \ge \underline S^w_k\mathbf 1_{Q_n}, \text{ with } \underline S^w_k:=\sum_{1\le i\le k}\underline \xi^w_i.
\end{equation}
Straightforward computations show that 
\[
\mathbb E[\underline\xi^w_1]= (1-\delta_n)n_wp=(1-\delta_n)\gamma^{(n)}_{\theta}, \quad \Var(\underline \xi^w_1) \le \gamma^{(n)}_{\theta}.
\]
Combined with Doob's maximal inequality and the fact that $\gamma^{(n)}_{\theta}\to \gamma_{\theta}$ under~\eqref{hyp: theta}, this implies that for any $\vep>0$, 
\[
\mathbb P\Big(\sup_{j\le \kappa_n}|\underline S^w_j-\gamma_{\theta}j|\ge \vep \kappa_n\Big) \le \mathbb P\Big(\sup_{j\le \kappa_n}|\underline S^w_j-(1-\delta_n)\gamma^{(n)}_{\theta}j|\ge \frac{\varepsilon \kappa_n}{2}\Big) +\mathbf 1_{\{|(1-\delta_n)\gamma^{(n)}_{\theta}-\gamma_{\theta}|\ge \varepsilon/2\}}\to 0
\]
as $n\to\infty$. Together with~\eqref{bd: gevent} ,~\eqref{bd: Sw2}, and the fact that $\eta_n \leq \kappa_n$ this yields 
\begin{align*}
 \mathbb P\Big(\inf_{j\le \eta_n}\Big(S^w_j-\gamma_{\theta}j\Big)\le -\vep\kappa_n\Big) &\le \mathbb P\Big(\inf_{j\le \eta_n}\Big(S^w_j-\gamma_{\theta}j\Big)\le -\vep\kappa_n; Q_n\Big)+1-\mathbb P(Q_n) \\
& \le \mathbb P\Big(\inf_{j \le \kappa_n}\Big(\underline S^w_j-\gamma_{\theta}j\Big)\le -\vep\eta_n\Big) + 1-\mathbb P(Q_n) 
\end{align*}
tending to 0 as $n\to\infty$. Combined with~\eqref{bd: Sw1}, and the fact that $\eta_n \leq \kappa_n$ this completes the proof.
\end{proof}

We next derive an upper bound for $I^b(|\Sigma^b(\kappa_n)|)$.

\begin{Lem}
\label{lem: Ib}
Let $\delta_n=\sqrt{\kappa_n/n}$. We have 
\begin{equation}
\label{bd: theta}
\liminf_{n\to\infty}\mathbb P\Big(K^b_{|\Sigma^b(\kappa_n)|+1}\ge (1-\delta_n)n_b\Big) =1.
\end{equation}
Moreover, for any $\vep>0$, the following convergence holds: 
\begin{equation}
\label{bd: Ika}
\limsup_{n\to\infty}\mathbb P\big(I^b(|\Sigma^b(\kappa_n)|)\ge \vep \kappa_n\big) = 0. 
\end{equation}
\end{Lem}

\begin{proof}
Let us denote $\eta_n=|\Sigma^b(\kappa_n)|$. 
Applying the upper bound in Lemma~\ref{lem: upbd} and the fact that $J^w(\eta_n)\le |\Sigma(\kappa_n)|=\kappa_n$, we can find two sequences of random variables $(\bar\xi^b_i)_{i\in [n_b]}$ and $(\chi_i)_{i\in[n_b]}$ such that  $(\chi_i)_{i\in [n_b]}$ is i.i.d.~$\cB(n_b, p)$ and given $(O^w_i)_{i\in [n_b]}$, $\{\bar\xi^b_i: i\in [n_b]\}$ is an independent collection with respective distributions $\cB(O^w_i\cdot n_b, p)$, as well as 
\begin{equation}
\label{bd: Sbup1}
\forall\,k\le \eta_n: \quad 0\le S^b_k\le S^b_{\eta_n}  \le \sum_{i\in [\kappa_n]}\bar\xi^b_i+\sum_{i\in [\kappa_n]}\chi_i.
\end{equation}
where we also used the fact that $\eta_n \leq \kappa_n$. Let us denote $\hat\gamma^{(n)}_{\theta}=n_bp=\sqrt{n_b/n_w}$, which converges to $\gamma_{\theta}^{-1}$ under Assumption~\eqref{hyp: theta}. 
Recall that $S^w_k=\sum_{i\le k}O^w_i$. The distribution  of $(\bar\xi^b_i)_{i\ge 1}$ implies that 
\[
\mathbb E\Big[\sum_{i\in [\kappa_n]}\bar\xi^b_i\,\Big|S^w_{\kappa_n}\Big] = \hat\gamma_{\theta}^{(n)}\cdot S^w_{\kappa_n}, \quad \Var\Big(\sum_{i\in [\kappa_n]}\bar\xi^b_i\,|\,S^w_{\kappa_n}\Big)\le \hat\gamma^{(n)}_{\theta}\cdot S^w_{\kappa_n}.
\]
Markov's inequality then yields that for any $\vep >0$, 
\begin{equation}
\label{bd: Sb11}
\mathbb P\Big(\sum_{i\in [\kappa_n]}\bar\xi^b_i -  \hat\gamma^{(n)}_{\theta}\cdot S^w_{\kappa_n} \ge \vep \kappa_n\,\Big|S^w_{\kappa_n}\Big) \le \frac{\hat\gamma^{(n)}_{\theta}\cdot S^w_{\kappa_n}}{\vep^2\kappa_n^2}\xrightarrow{n\to\infty} 0,
\end{equation}
since $S^w_{\kappa_n}=O_p(\kappa_n)$ according to ~\eqref{bd: Sw1}. 
A similar argument shows that 
\begin{equation}
\label{cv: chi}
\mathbb P\Big(\sum_{i\in [\kappa_n]}\chi_i \ge  (\hat\gamma^{(n)}_{\theta}+ \vep )\kappa_n\Big)\xrightarrow{n\to\infty}0. 
\end{equation}
Combining this with~\eqref{bd: Sb11}, the convergence of $\hat\gamma^{(n)}_{\theta}\to\gamma_{\theta}^{-1}$ and the convergence of $S^w_k$ in Proposition~\ref{prop: LLN-w}, we deduce that 
\[
\sum_{i\in [\kappa_n]}\bar\xi^b_i + \sum_{i\in [\kappa_n]}\chi_i = O_p(\kappa_n), \quad \text{as }n\to\infty. 
\]
Thanks to~\eqref{bd: Sbup1}, this implies $
S^b_{\eta_n} = O_p(\kappa_n)$ as $n\to\infty$. 

Note that $\kappa_n/n_b=o(\delta_n)$, as $n\to\infty$, as well as $I^b(k)\le k$ for all $k$. Recall from~\eqref{id: Kset} the disjoint decomposition of $V^b_n\setminus\cK^b_{k}$. The above then implies  
\[
n_b-K^b_{\eta_n+1}=S^b_{\eta_n}+I^b(\eta_n+1)\le S^b_{\eta_n}+\eta_n+1 = O_p(\kappa_n)=o_p(\delta_n\cdot n_b), \quad\text{as }n\to\infty,
\]
from which~\eqref{bd: theta} follows. We now use this to construct a lower bound for $S^b$, which will then imply the upper bound~\eqref{bd: Ika} for $I^b(\eta_n)$.  
Applying Lemma~\ref{lem: lowbd} with $\underline n_b=(1-\delta_n)n_b$, we can find a sequence of independent variables 
$(\underline \xi^b_i)_{i\in [n_b]}$ with respective distributions $\cB((1-\delta_n)O^w_i n_b, p)$. Moreover, setting $\underline S^b_k=\sum_{i\le k} \underline \xi^b_i$, we have 
\begin{equation}
\label{bd: Sb}
\forall k\in [n_b]: \quad S^b_k\ge \underline S^b_k\mathbf 1_{\{K^b_{k+1}\ge (1-\delta_n)n_b\}}.
\end{equation}
We also note that $\{\underline S^b_k-(1-\delta_n)S^w_k \hat\gamma^{(n)}_{\theta}: k\in [n_b]\}$  is a martingale with respect to the natural filtration of $(S^w_i)_{i\ge 1}$, and we have 
\[
\Var\big(\underline S^b_k\,|\,(O^w_i)_{i\le k}\big) = \sum_{i\le k}\Var\big(\underline\xi^b_i\,|\,O^w_i\big) \le \hat\gamma^{(n)}_{\theta} \cdot S^w_k.
\]
We then deduce from Doob's maximal inequality and the fact that $S^w_{\kappa_n}=O_p(\kappa_n)$  that for any $\vep>0$, 
\[
\mathbb P\Big(\sup_{k\le \kappa_n}|\underline S^b_k- \hat\gamma^{(n)}_{\theta}\cdot S^w_k \big| \ge \vep \kappa_n\Big) \xrightarrow{n\to\infty }0. 
\]
Combined with Proposition~\ref{prop: LLN-w} and the convergences $\delta_n\to0$, $\hat\gamma^{(n)}_{\theta}\to \gamma_{\theta}^{-1}$, this implies that
\begin{equation}
\label{bd: loSbk}
\mathbb P\Big(\sup_{k\le \eta_n}\big|\underline S^b_k- k\big|\ge \vep \kappa_n\Big) \xrightarrow{n\to\infty}0. 
\end{equation}
Hence, we have 
\begin{align}\notag
&\quad\;\mathbb P\Big(\inf_{k\le \eta_n}\big(S^b_k-k\big)\le -\vep\kappa_n\Big) \\ \notag
&\le \mathbb P\Big( \inf_{k\le \eta_n}\big(S^b_k-k\big) \le -\vep\kappa_n; K^b_{\eta_n+1}\ge (1-\delta_n)n_b\Big) + \mathbb P\big(K^b_{\eta_n+1}< (1-\delta_n)n_b\big) \\ \label{Sbklwbd}
& \le \mathbb P\Big( \inf_{k\le \eta_n}\big(\underline S^b_k-k\big) \le -\vep\kappa_n\Big) + \mathbb P\big(K^b_{\eta_n+1}< (1-\delta_n)n_b\big)\xrightarrow{n\to\infty}0,
\end{align}
as a result of~\eqref{bd: loSbk} and~\eqref{bd: theta}. The bound for $I^b(\eta_n)$ now follows as a result of~\eqref{id: I}. 
\end{proof}

We explain here how Proposition~\ref{prop: bound-compo} is implied by Lemma~\ref{lem: Ib}. 

\begin{proof}[Proof of Proposition~\ref{prop: bound-compo}]
We utilise the coupling described in the paragraph {\bf the dual problem}. Let $\cJ^w$ denote the set of white lead vertices in $G(n_b, n_w, p)$; in particular, we have
\[
J^w(k) = |\cJ^w\cap \Sigma(\tau^b_k)|. 
\]
We then denote by $\hat\cJ^b$ the set of black lead vertices in $\hat G(n_b, n_w, p)$. Since a black lead vertex is also a root vertex, we have 
\[
\big|\hat \cJ^b\cap \hat\Sigma\big(\hat\tau^b_{|\hat\Sigma^b(\kappa_n)|}\big)\big|\le \hat I^b\big(\big|\hat\Sigma^b(\kappa_n)\big|\big),
\]
where $\hat I^b(k)$ stands for the set of root vertices up to the rank $\hat\tau^b_k$ in $\hat G(n_b, n_w, p)$. 
We note that $(\hat K_{n_b, n_w})_{n\ge 1}$ is itself a sequence of complete bipartite graphs that satisfies the assumptions of Lemma~\ref{lem: Ib}, with the proportion of black vertices converging to $\hat\theta =1-\theta$. In consequence, we can apply Lemma~\ref{lem: Ib} to find 
\[
\big|\hat \cJ^b\cap \hat\Sigma\big(\hat\tau^b_{|\hat\Sigma^b(\kappa_n)|}\big)\big|\le \hat I^b\big(\big|\hat\Sigma^b(\kappa_n)\big|\big)=o_p(\kappa_n), \quad n\to\infty. 
\]
Thanks to the coupling, we then have 
\begin{equation}
\label{cv: Jw}
\big|\cJ^w\cap\Sigma\big(\tau^w_{|\Sigma^w(\kappa_n)|}\big)\big|=\big|\hat \cJ^b\cap \hat\Sigma\big(\hat\tau^b_{|\hat\Sigma^b(\kappa_n)|}\big)\big|=o_p(\kappa_n), \quad n\to\infty. 
\end{equation}
Finally, let us show that 
\begin{equation}
\label{bd: J}
J^w(|\Sigma^b(\kappa_n)|)= \big|\cJ^w\cap\Sigma\big(\tau^b_{|\Sigma^b(\kappa_n)|}\big)\big| \le \big|\cJ^w\cap\Sigma\big(\tau^w_{|\Sigma^w(\kappa_n)|}\big)\big|. 
\end{equation}
Indeed, the first is straightforward from the definition of $J^w$. For the inequality above, let us consider the colour of $\sigma(\kappa_n)$. If this is a white vertex, then we necessarily have
\[
\tau^w_{|\Sigma^w(\kappa_n)|}=\kappa_n \quad \text{and}\quad \tau^b_{|\Sigma^b(\kappa_n)|}<\tau^w_{|\Sigma^w(\kappa_n)|},
\]
so that the bound in~\eqref{bd: J} follows. In the opposite case, we have $\tau^w_{|\Sigma^w(\kappa_n)|}\le \kappa_n-1$; moreover, there is no white vertex between the ranks $\tau^w_{|\Sigma^w(\kappa_n)|}+1$ and $\tau^b_{|\Sigma^b(\kappa_n)|}$. Hence, the inequality in~\eqref{bd: J} still holds, which is in fact an equality in this case.  

The conclusion now follows from~\eqref{bd: J} and~\eqref{cv: Jw}. 
\end{proof}

Before ending this subsection, we prove an analogue of Proposition~\ref{prop: LLN-w} for $S^b$. 
\begin{Prop}[Law of large numbers for $S^b$]
\label{prop: LLN-b}
For all $\vep>0$, the following convergence holds:
\begin{equation}
\label{cv: Sb}
\mathbb P\Big(\sup_{k\le |\Sigma^b(\kappa_n)|}\Big|S^b_k-k\Big| \ge \vep \kappa_n\Big) \xrightarrow{n\to\infty} 0.
\end{equation}
\end{Prop}

\begin{proof}
We shorthand $\eta_n=|\Sigma^b(\kappa_n)|$. 
In view of~\eqref{Sbklwbd}, it suffices to show that 
\begin{equation}
\label{cv: sb1}
\mathbb P\Big(\sup_{k\le \eta_n}\big(S^b_k-k\big) \ge \vep \kappa_n\Big) \xrightarrow{n\to\infty} 0.
\end{equation}
To this end, we improve the upper bound in~\eqref{bd: Sbup1} with 
\begin{equation}
\label{bd: Sbup2}
\forall\,k\le \eta_n: \quad S^b_k \le \bar S^b_k + \hat S^b_{\eta_n}, \ \text{where }\bar S^b_k:=\sum_{i\le k}\bar\xi^b_i,\ \hat S^b_{\eta_n}:= \sum_{j\le J^w(\eta_n)}\chi_j, 
\end{equation}
where $(\bar\xi^b_i)_{i\ge 1}$ and $(\chi_i)_{i\ge 1}$ are as in Lemma~\ref{lem: upbd}.  
Denote $\hat\gamma^{(n)}_{\theta}=n_bp = \sqrt{n_b/n_w}$. 
We note that $\bar S^b_k-\hat\gamma^{(n)}_{\theta}\cdot S^w_k, k\ge 1$, is a martingale with respect to the natural filtration of $(S^w_k)_{k\ge 1}$. Straightforward computations then yield:
\[
\mathbb E\big[\bar S^b_k\,\big|\, S^w_k\big] = \hat\gamma^{(n)}_{\theta}S^w_k\quad\text{and}\quad \Var\big(\bar S^b_k\,\big|\, S^w_k\big) \le \hat\gamma^{(n)}_{\theta}S^w_k. 
\]
Proceeding as previously, we obtain from Doob's maximal inequality that for any $\vep>0$, 
\[
\mathbb P\Big(\sup_{k\le \kappa_n}\big|\bar S^b_k-\hat\gamma^{(n)}_{\theta}S^w_k\big|\ge \vep\kappa_n\Big) \xrightarrow{n\to\infty}0.
\]
Combined with Proposition~\ref{prop: LLN-w}, this implies that 
\begin{equation}
\label{cv: barS}
\mathbb P\Big(\sup_{k\le \eta_n}\big|\bar S^b_k-k\big|\ge \tfrac12\vep\kappa_n\Big) \xrightarrow{n\to\infty}0.
\end{equation}
On the other hand, similar to the bound~\eqref{cv: chi}, we find that for $\delta>0$, 
\[
\mathbb P\Big(\sum_{i\le \delta \kappa_n}\chi_i \ge 2\hat\gamma^{(n)}_{\theta}\delta\kappa_n\Big)\xrightarrow{n\to\infty}0. 
\]
Combined with Proposition~\ref{prop: bound-compo} and the convergence $\hat\gamma^{(n)}_{\theta}\to \gamma_{\theta}^{-1}$, this yields for $\delta \in (0, \vep\gamma_{\theta}/4)$, 
\[
\mathbb P\Big(\hat S^b_{\eta_n}\ge \tfrac12\vep\kappa_n\Big) \le \mathbb P\Big(J^w(\eta_n)\ge \delta \kappa_n\Big) + \mathbb P\Big(\sum_{i\le \delta\kappa_n}\chi_i\ge \tfrac12\vep \kappa_n\Big)\xrightarrow{n\to\infty}0.
\]
Together with~\eqref{bd: Sbup2} and~\eqref{cv: barS}, this proves~\eqref{cv: sb1}, which then completes the proof. 
\end{proof}

\subsection{Proof of Proposition~\ref{prop: LLN}}
\label{sec: limit-sub}

Recall respectively from~\eqref{def: Ab} and~\eqref{def: Aw} the sets $\cA^b_k$ and $\cA^w_k$, and recall that $A^b(k)$ and $A^w(k)$ are their respective sizes. Recall from Lemma~\ref{lem: Aset} that $R(k)=S^w_{k-1}-A^w(k)$. We shall prove Proposition~\ref{prop: LLN} by combining Proposition~\ref{prop: LLN-w} with an upper bound on $A^w(k)$.

\begin{Lem}
\label{lem: Ab}
For all $\vep>0$, we have
\[
\mathbb P\Big(\sup_{k\le |\Sigma^b(\kappa_n)|}A^b(k)\ge \vep \kappa_n\Big) \xrightarrow{n\to\infty}0.
\]
\end{Lem}

\begin{proof}
According to Lemma~\ref{lem: Aset}, $A^b(k) = S^b_k-k + I^b(k)$. The conclusion is now immediate from~\eqref{cv: Sb} and~\eqref{bd: Ika}. 
\end{proof}

\begin{Lem}
\label{lem: AbAw}
Let $\vep\in (0, \gamma_{\theta}^{-1}\wedge 1)$. We have
\[
\liminf_{n\to\infty}\inf_{k\le \kappa_n}\mathbb P\big(A^b(k)\ge \vep^2\kappa_n\,\big|\, A^w(k)\ge \vep \kappa_n\big) =1. 
\]
\end{Lem}

\begin{proof}
Let us first show that 
\begin{equation}
\label{eq: NA-set}
\cN\big(\cA^w_k\big)\setminus\Sigma^b(\tau^b_k)\subseteq \cA^b_k. 
\end{equation}
Let $u\in\cN(\cA^w_k)\setminus \Sigma^b(\tau^b_k)$.  
Since $\cA^w_k\subseteq \cup_{j\le k-1}\cO^w_j$ and $\cO^w_j\subseteq \cN(\sigma^b(j))$, we can find some $j\le k-1$ and $v\in \cN(\sigma^b(j))$ so that $u$ is adjacent to $v$. It follows that 
\[
u\in \bigcup_{j\le k-1}\cN^2(\sigma^b(j))\setminus \Sigma^b(\tau^b_k) \subseteq \cN^2\big(\Sigma^b(\tau^b_{k-1})\big)\setminus\Sigma^b(\tau^b_k)\subseteq \bigcup_{j\le k-1}\cO^b_j\setminus\Sigma^b(\tau^b_k)\subseteq\cA^b_k,
\]
where we have used Lemma~\ref{lem: Kset} {\it (iv)} in the second inclusion.

We introduce the following partition of $\cN(\cA^w_k)\setminus\Sigma^b(\tau^b_k)$. Assume that $\cA^w_k=\{u_1, u_2, \cdots, u_{\ell}\}$, where the $u_i$'s are ranked in a deterministic but arbitrary manner. We then set $\cX_k(1)=\cN(u_1)\setminus \Sigma^b(\tau^b_k)$ and $X_i=|\cX_k(1)|$, and more generally for $2\le i\le A^w(k)$, 
\[
\cX_k(i)=\cN(u_i)\setminus\Big(\bigcup_{i'<i}\cX_k(i')\bigcup\Sigma^b(\tau^b_k)\Big), \quad X_i = |\cX_k(i)|. 
\]
By definition, $\cX_k(i), 1\le i\le \ell$, are disjoint and their union is precisely $\cN(\cA^w_k)\setminus \Sigma^b(\tau^b_k)$. Together with~\eqref{eq: NA-set}, this yields that for each $k\in [n_b]$, 
\begin{equation}
\label{id: count-X}
X:=\sum_{1\le i\le A^w(k)}X_i \le A^b(k). 
\end{equation}
Recall from~\eqref{def: D-dist} the distribution $\cD(n, k, p)$. Next, let us show that given $A^w(k)=\ell$,
\begin{equation}
\label{dist: X}
X=\sum_{1\le i\le \ell} X_i \text{ follows the distribution }\cD(n_b-k, \ell, p).
\end{equation}
The arguments are similar to those in the proof of Proposition~\ref{prop: incre}; so we only outline the main steps. 
Recall the edge set $F_k$ and the sigma-algebra $\cG^w_k=\cF_{\tau^b_k-1}$ from Section~\ref{sec: app}. 
Since $\cA^w_k=\cup_{j\le k-1}\cO^w_j\setminus \Sigma^w(\tau^b_k)$, we deduce from
Lemma~\ref{lem: Gk}  that $\cA^w_k$ is $\cG^w_k$-measurable. 
Moreover, 
on the event $\{u_1\in \cA^w_k\}$, 
the edge set $\{e=\{u_1, v\}: v\in V^b_n\setminus\Sigma^b(\tau^b_k)\}$ is a measurable subset of $F_{\tau^b_k-1}$. We then apply Lemma~\ref{lem: Fk} {\it (iii)} to find that $X_1\sim \cB(n_b-k, p)$. As in the proof of Proposition~\ref{prop: incre}, we can express $\cX_k(i)$ as a function of $\{\{u_i, v\}: v\in V^b_n\setminus \Sigma^b(\tau^b_k)\}$ and $\cX_k(i'), i'<i$, which then implies~\eqref{dist: X}. 

To ease the writing, we denote $N=n_b-k$ and $L=\lfloor \varepsilon\kappa_n\rfloor$. 
We deduce from the assumptions on $\kappa_n$, the condition~\eqref{hyp: theta} and the fact that $p=1/\sqrt{n_bn_w}$ that
\begin{equation}
\label{cv: Np}
np\xrightarrow{n\to\infty} \frac{1}{\sqrt{\theta(1-\theta)}} \quad\text{and}\quad \max_{k\le\kappa_n}\big|Np-\gamma_{\theta}^{-1}\big| \xrightarrow{n\to\infty}0. 
\end{equation}
Take $\delta'_n=\sqrt{1/(n\kappa_n)}$. We note that
\begin{equation}
\label{eq: delta'n}
\delta'_n \kappa_n \xrightarrow{n\to\infty} 0 \quad\text{and}\quad n\delta'_n  \xrightarrow{n\to\infty} \infty.
\end{equation}
Let $(X'_i)_{i\ge 1}$ be an i.i.d.~sequence of random variables with common distribution $\cB(N, p)$, and let $X'=\sum_{1\le i\le  L}X'_i$. Since each $X_i$ is 
stochastically bounded by $X'_i$, conditionally on $(X_j)_{j<i}$, we deduce that for $n$ sufficiently large, 
\begin{align*}
\sup_{k\le\kappa_n}\mathbb P\Big(\sum_{1\le i\le L} X_i\ge \varepsilon N\kappa_n\delta'_n\,\Big|\,A^w(k)\ge \varepsilon \kappa_n\Big)
&\le \sup_{k\le \kappa_n}\mathbb P\big(X'\ge \varepsilon N\kappa_n\delta'_n\big)\\
&\le \sup_{k\le \kappa_n}\mathbb P\big(\big|X'- \mathbb E[X']\big|\ge \varepsilon N \kappa_n\delta'_n-LNp\big) \\
&\le \sup_{k\le \kappa_n}\frac{1}{N}\cdot\frac{Lp}{(\varepsilon  \kappa_n\delta'_n-Lp)^2}\xrightarrow{n\to\infty} 0,
\end{align*}
since $N\ge n_b-\kappa_n\to \infty$ and $Lp=o(\kappa_n\delta'_n)$ as $n\to\infty$ as a result of~\eqref{eq: delta'n}.  
Denote by $Q_{n, j}$ the event that $\sum_{1\le i\le j}X_i< \varepsilon N\kappa_n\delta'_n$. We have shown that
\begin{equation}
\label{eqbd: En}
\sup_{k\le \kappa_n}\mathbb P\big(Q_{n,L}^c\,\big|\, A^w(k)\ge \vep\kappa_n\big)\xrightarrow{n\to\infty}0.
\end{equation}
Given $\sum_{1\le i\le j-1}X_i$, $X_{j}$ has the distribution $\cB(N-\sum_{1\le i\le j-1}X_i, p)$. Arguing as in the proof of Lemma~\ref{lem: lowbd}, we can find a sequence $(X''_i)_{i\ge 1}$ such that given $A^w(k)$, it is i.i.d.~with common distribution $\cB(N-\varepsilon N\kappa_n\delta'_n, p)$, and satisfies 
\[
X_i\ge X''_i\mathbf 1_{Q_{n, i-1}}, \quad 1\le i\le L.
\]
It follows that on the event $\{A^w(k)\ge \varepsilon \kappa_n\}$, 
\begin{equation}
\label{eq: lwX''}
X\ge \sum_{1\le i\le L}X''_i\mathbf 1_{Q_{n, i-1}}\ge \sum_{1\le i\le L}X''_i\mathbf 1_{Q_{n, L}}
\end{equation}
On the other hand, we deduce from~\eqref{cv: Np} and~\eqref{eq: delta'n} that 
\[
\frac{L(1-\varepsilon \kappa_n\delta'_n)Np}{\big(L(1-\varepsilon \kappa_n\delta'_n)Np-\varepsilon^2 \kappa_n\big)^2}\sim \frac{\varepsilon \kappa_n\gamma^{-1}_{\theta}}{(\varepsilon \kappa_n\gamma^{-1}_{\theta}-\varepsilon^2\kappa_n)^2}, 
\]
which tends to 0 as $n\to\infty$ as $\varepsilon<\gamma_{\theta}^{-1}$. 
Moreover, as the convergence of $Np$ in~\eqref{cv: Np} is uniform in $k$, this convergence also hold uniformly for $k\le \kappa_n$. 
Let $X'':=\sum_{1\le i\le L}X''_i$. We have 
\begin{align}\notag
\sup_{k\le \kappa_n}\mathbb P\big(X'' < \varepsilon^2 \kappa_n\big) &\le \sup_{k\le \kappa_n}\mathbb P\big(\big|X''-\mathbb E[X'']\big|\ge L(1-\varepsilon \kappa_n\delta'_n)Np-\varepsilon^2 \kappa_n\big)\\ \label{bd: X''}
&\le \sup_{k\le \kappa_n}\frac{L(1-\varepsilon \kappa_n\delta'_n)Np}{\big(L(1-\varepsilon \kappa_n\delta'_n)Np-\varepsilon^2 \kappa_n\big)^2}\xrightarrow{n\to\infty}0. 
\end{align}
Combined with~\eqref{eq: lwX''} and~\eqref{eqbd: En}, this implies that
\begin{align*}
&\quad\sup_{k\le\kappa_n}\mathbb P\big(X< \vep^2 \kappa_n\,\big|\,A^w(k)\ge \vep\kappa_n \big) \\
&\le \sup_{k\le \kappa_n}\mathbb P\big(X< \vep^2\kappa_n; Q_{n,L}\,\big|\, A^w(k)\ge \vep\kappa_n)+\sup_{k\le\kappa_n}\mathbb P\big(R^c_{n, L}\,\big|\,A^w(k)\ge \vep\kappa_n\big) \\
&\le \sup_{k\le\kappa_n}\mathbb P\big(X''< \vep^2\kappa_n\,\big|\, A^w(k)\ge \vep\kappa_n) +\sup_{k\le\kappa_n}\mathbb P\big(Q_n^c\,\big|\,A^w(k)\ge \vep\kappa_n\big)\xrightarrow{n\to\infty}0.
\end{align*}
Together with~\eqref{id: count-X}, this concludes the proof. 
\end{proof}

\begin{Lem}
\label{lem: O-R}
For all $\vep>0$, we have
\[
\mathbb P\Big(\sup_{k\le |\Sigma^b(\kappa_n)|}A^w(k)\ge \vep \kappa_n\Big) \xrightarrow{n\to\infty}0.
\]
\end{Lem}

\begin{proof}
It suffices to prove the statement for $\vep<\gamma^{-1}_{\theta}$. Writing $\eta_n=|\Sigma^b(\kappa_n)|$, we note that 
\begin{align*}
&\;  \mathbb P\Big(\max_{k\le\eta_n}A^b(k)\ge \vep^2 \kappa_n; \max_{k\le \eta_n}A^w(k)\ge \vep \kappa_n\Big)
\ge \mathbb P\Big(\exists\,k\le\eta_n: A^b(k)\ge \vep^2\kappa_n; A^w(k)\ge \vep\kappa_n\Big) \\
& \ge \mathbb P(\exists\, k\le\eta_n: A^w(k)\ge \vep\kappa_n\big)\cdot \inf_{k\le \kappa_n}\mathbb P(A^b(k)\ge \vep^2\kappa_n\,|\,A^w(k)\ge \vep \kappa_n)
\end{align*}
Lemma~\ref{lem: AbAw} implies that for $n$ sufficiently large, 
\[
\inf_{k\le \kappa_n}\mathbb P(A^b(k)\ge \vep^2\kappa_n\,|\,A^w(k)\ge \vep \kappa_n)\ge \vep \kappa_n\Big)\ge \frac{1}{2}.
\]
It follows that for those $n$, 
\[
\mathbb P\Big(\max_{k\le\eta_n}A^w(k)\ge \vep \kappa_n\Big)\le 2\mathbb P\Big(\max_{k\le\eta_n}A^b(k)\ge \vep^2 \kappa_n; \max_{k\le \eta_n}A^w(k)\ge \vep \kappa_n\Big)\le 2\mathbb P\Big(\max_{k\le\eta_n}A^b(k)\ge \vep^2\kappa_n\Big).
\]
Thanks to Lemma~\ref{lem: Ab}, the last term above tends to 0. The conclusion follows.
\end{proof}

\begin{proof}[Proof of Proposition~\ref{prop: LLN}]
This is immediate from Lemma~\ref{lem: Aset}, Proposition~\ref{prop: LLN-w} and Lemma~\ref{lem: O-R}. 
\end{proof}

\appendix

\section{Some measurability properties of the graph exploration}
\label{sec: app}

We collect here some technical results that are used in proofs. 
Recall that $E_n$ stands for the edge set of the complete bipartite graph $K_{n_b, n_w}$. For $0\le j\le n$, we define 
\[
F_j=\big\{e=\{u, v\}\in E_n: u\in V^b_n\setminus\Sigma^b(j), v\in V^w_n\setminus\Sigma^w(j)\big\}. 
\]
Then $F_j^c:=E_n\setminus F_j$ contains all the edges that are adjacent to some $\sigma(i), i\in [j]$. Let $\cF_0$ be the sigma-algebra generated by $\sigma(1)$; for $j\ge 1$, let $\cF_j$ be the sigma-algebra generated by 
\[
\big(\sigma(i): i\in [j]\big) \quad\text{and}\quad \big\{U_e: e\in F^c_j\big\}.
\]
Observe that $\cF_0\subseteq \cF_1\subseteq \cF_2\subseteq\cdots$. 

\begin{Lem}
\label{lem: Fk}
The following statements hold true for each $0\le j\le n-1$:
\begin{enumerate}[(i)]
\item 
we have $\sigma(j+1)\in \cF_{j}$;
\item
both $(e_i)_{i \in [j]}$ and $(U_{e_i})_{i\in [j]}$ are $\cF_{j}$-measurable.
\item 
 conditional on $\cF_j$, the collection
$\{U_e: e\in F_j\}$ is distributed as i.i.d.~uniform random variables on $(0, 1)$. 
\end{enumerate}
\end{Lem}

\begin{proof}
We fix $n\in \N$ and argue by induction on $j$. For $j=0$, we have $\sigma(1)\in\cF_0$ by definition, and the two sequences in {\it (ii)} are empty. Moreover, we have $F_0=E_n$, and the collection $\{U_e: e\in E_n\}$ is independent of $\sigma(1)$, and thus of $\cF_0$. Suppose that the claim holds for some $j-1$ and let us argue it also holds for $j\ge 1$. We recall from Algorithm~\ref{algo: Prim} that $\sigma(j+1)$ is determined by the weights of the edge set $\{\{u, v\}: u\in \Sigma(j), v\in V_n\setminus\Sigma(j)\}$. As a result, $\sigma(j+1)$ is a function of $\{U_e: e\in F^c_{j}\}$ and $(\sigma(i))_{i\in [j]}$; hence $\sigma(j+1)\in \cF_{j}$.  Similarly, we can argue that $e_j$ and $U_{e_j}$ are determined from $\{U_e: e\in F^c_j\}$ and $(\sigma(i))_{i\in [j]}$, and therefore $\cF_j$-measurable.
The claim for $j-1$ states that given $\cF_{j-1}$,  $\{U_e: e\in F_{j-1}\}$ are i.i.d.~uniform. Since $\sigma(j)\in\cF_{j-1}$, 
this implies that
given $(\sigma(i))_{i\in [j]}$ and $\{U_e: e\in F^c_{j-1}\}$,  the same collection are i.i.d.~uniform. Next, let us denote $D_j=F_{j-1}\setminus F_j$, so that $D_j, F_j$ are disjoint subsets of $F_{j-1}$. 
We note that $D_j$ consists of all the edges $e=\{\sigma(j), v\}$ with $v\in V_n\setminus\Sigma(j)$; consequently, 
both $D_j$ and $F_j$ are functions of $(\sigma(i))_{i\in [j]}$. It follows that given $(\sigma(i))_{i\in [j]}$ and $\{U_e: e\in F^c_{j-1}\}$, $\{U_e: e\in F_{j}\}$ is a collection of i.i.d.~uniform variables and is independent of $\{U_e: e\in D_j\}$. 
Finally, since $F_j^c=F_{j-1}^c\cup D_j$, we conclude that given $\cF_j$, $\{U_e: e\in F_{j}\}$ are i.i.d.~uniformly distributed.
\end{proof}

An integer-valued random variable $T$ is said to be a {\em stopping time} if $\{T\le j\}\in \cF_j$ for all $j\ge 0$. 
Recall that $\tau^b_{k}$ is the rank of the $k$th black vertex in $(\sigma(i))_{i\in [n]}$. For $k\in [n_b]$, we note that $\tau^b_{k}-1$ is a stopping time, since 
\[
\{\tau^b_{k}\le j+1\}= \big\{\big|V^b_n\cap\{\sigma(i): 1\le i\le j+1\}\big|\ge k\big\},
\]
which is $\cF_j$-measurable according to the previous lemma. Then we can define $\cG^w_k:=\cF_{\tau^b_{k}-1}$ as the sigma-algebra generated by the events 
$A\cap \{\tau^b_{k}-1\le j\}$ for all $A\in \cF_j, j\ge 0$. 
Recall from~\eqref{def: Ow} the sets $\cO^w_k$ and $\cK^w_k$ of the random graph $G(n_b, n_w, p)$.   
Define 
\begin{equation}
\label{def: Gwkset}
G^w_k=\big\{e=\{u, v\}\in E_n: u\in V^b_n\setminus\Sigma^b(\tau^b_{k-1}), v\in \cK^w_k\big\}, 
\end{equation}
with the convention $\tau^b_0=0$.
\begin{Lem}
\label{lem: Gk}
Let $p\in [0, 1]$. The following statements hold true for $G(n_b, n_w, p)$ and $k\in [n_b]$. 
\begin{enumerate}[(i)]
\item 
$(\sigma(i))_{i\le \tau^b_k}$ and $\tau^b_k$ are $\cG^w_k$-measurable;
\item 
$\cN(\sigma(\tau^b_{k-1}))$, $(\cO^w_i)_{i\in [k-1]}$, $(\cK^w_i)_{i\in [k]}$ and $G^w_k$ are $\cG^w_k$-measurable;
\item 
conditional on $\cG^w_k$, $\{U_e: e\in G^w_k\}$ are i.i.d.~uniformly distributed on $(0, 1)$.
\end{enumerate}
\end{Lem}

\begin{proof}
We first note that since $\{\tau^b_k-1=j\}=\{\tau^b_k-1\le j\}\cap\{\tau^b_k-1>j-1\}$ and $\{\tau^b_k-1\le j\}=\cup_{i\le j}\{\tau^b_k-1=i\}$, $\cG^w_k$ is equivalently generated by the events $A\cap \{\tau^b_{k}-1 = j\}$ for all $A\in \cF_j, j\ge 0$. It then follows from Lemma~\ref{lem: Fk} that for all $j\ge 0$, all sets $B\subset \N^{j+1}$ and $B'\subset \N$, 
\[
\{(\sigma(i))_{i\le j+1}\in B\}\cap \{\tau^b_k=j+1\},\quad \{\tau^b_k\in B'\}\cap \{\tau^b_k=j+1\}
\]
are both $\cF_j$-measurable. Comparing this with the definition of $\cG^w_k$, we see that {\em (i)} holds. 

For {\em (ii)} and {\em (iii)}, we condition on the value of $\tau^b_k$ and argue 
by an induction on $k$. For the case $k=1$, we recall that $\tau^b_1\in\{1, 2\}$, as $\sigma(1)$ and $\sigma(2)$ must have opposite colours. In either case, we readily check that $\cK^w_1=V^w_n\setminus \Sigma^w(\tau^b_1-1)$ is $\cG^w_1$-measurable, and that $G^w_1=F_{\tau^b_1-1}$. The claim in {\em (iii)} then follows from Lemma~\ref{lem: Fk} {\em(iii)}.
For the induction step from $k-1$ to $k$, we can assume $\tau^b_{k-1}-1=j'\ge 0$ and $\tau^b_k-1=j>j'$, so that $\cG^w_{k-1}=\cF_{j'}$ and $\cG^w_k=\cF_j$. 
The proof of {\em (ii)} then reduces to showing that $\cN(\sigma(\tau^b_{k-1}))$, $\cO^w_{k-1}, \cK^w_k$ and $G^w_k$ are all $\cF_j$-measurable. 
To that end, we first note that the 1-neighbourhood $\cN(\sigma(\tau^b_{k-1}))$ is determined by the weights from the edge set
\[
\big\{e=\{\sigma(\tau^b_{k-1}), u\}: u\in V^w_n\}=\big\{e=\{\sigma(j'+1), u\}: u\in V^w_n\big\},
\]
which is a subset of $F^c_{j'+1}$. It follows that $\cN(\sigma(\tau^b_{k-1}))$ is $\cF_{j'+1}$-measurable. Since $\cK^w_{k-1}$ is $\cF_{j'+1}$-measurable by the induction assumption, we deduce that $\cO^w_{k-1}=\cN(\sigma(\tau^b_{k-1}))\cap \cK^w_{k-1}$ is also $\cF_{j'+1}$-measurable. As $\cF_{j'+1}\subseteq \cF_j$, this shows that both $\cN(\sigma(\tau^b_{k-1}))$ and $\cO^w_{k-1}$ are $\cF_j$-measurable. Since $\Sigma^w(\tau^b_k)=\Sigma^w(\tau^b_k-1)$, we also have that
\[
\cK^w_{k}=\cK^w_{k-1}\setminus\big(\cO^w_{k-1}\cup \Sigma^w(\tau^b_k-1)\big)=\cK^w_{k-1}\setminus\big(\cO^w_{k-1}\cup \Sigma^w(j)\big)
\]
is $\cF_j$-measurable. 
It is now also clear from its definition that $G_k^w$ is  $\cF_j$-measurable. This proves {\em (ii)}. For {\em (iii)}, since $\cK^w_k\subseteq V^w_n\setminus \Sigma^w(\tau^b_{k})=V^w_n\setminus \Sigma^w(\tau^b_{k}-1) $ and $\Sigma^b(\tau^b_{k-1})=\Sigma^b(\tau^b_k-1)$, we have 
\[
G^w_k\subseteq  \big\{e=\{u, v\}: u\in V^b_n\setminus\Sigma^b(\tau^b_k-1), v\in V^w_n\setminus\Sigma^w(\tau^b_k-1)\big\} = F_j.
\]
We have already seen that $G^w_k$ is $\cF_j$-measurable. 
The conclusion now follows from Lemma~\ref{lem: Fk}. 
\end{proof}

Since $\cO^w_k$ contains white vertices of ranks higher than $\tau^b_k$, for the exploration of $\cO^b_k$, we introduce the following filtration: for $k\in [n_b]$, let $\cG^b_k$ be the sigma-algebra generated by 
\[
\cG^w_k, \quad \cO^w_k \quad\text{and}\quad \big(\{U_e: e\in H_i\}\big)_{i\le k-1}. 
\]
where $H_i:=\{e=\{u, v\}: u\in V^b_n\setminus \Sigma^b(\tau^b_i) , v\in \cO^w_{i}\}$.
Note that $\cG^b_1\subseteq \cG^b_2\subseteq \cG^b_3\subseteq\cdots$. 
We also define
\begin{equation}
\label{def: Gbkset}
G^b_k = \big\{e=\{u, v\}: u\in V^b_n\setminus\Sigma^b(\tau^b_k), v\in \cK^w_k\big\}.
\end{equation}
\begin{Lem}
\label{lem: Gk2}
Let $p\in [0, 1]$. The following statements hold true for $G(n_b, n_w, p)$ and $k\in [n_b]$. 
\begin{enumerate}[(i)]
\item 
$(\cO^b_i)_{i\in [k-1]}$, $(\cK^b_i)_{i\in [k]}$ and $G^b_{k}$ are $\cG^b_k$-measurable;
\item 
conditional on $\cG^b_k$, $\{U_e: e\in G^b_k\}$ are i.i.d.~uniformly distributed on $(0, 1)$.
\end{enumerate}
\end{Lem}

\begin{proof}
Recall that $\sigma^b(k)=\sigma(\tau^b_k)$ is the $k$th black vertex in Prim's ranking. 
For {\em (i)}, we argue by an induction on $k$. For $k=1$, $\cG^b_1$ is the sigma-algebra generated by $\cG^w_1$ and $\cO^w_1$. In this case, $(\cO^b_i)_{i\in [k-1]}$ is an empty sequence and $\cK^b_1=V^b_n\setminus\{\sigma^b(1)\}$. We have seen in Lemma~\ref{lem: Gk} that $\sigma^b(1)$ is $\cG^w_1$-measurable and therefore $\cG^b_1$-measurable. This shows that both sequences in {\em(i)} are $\cG^b_1$-measurable. Moreover, since $\cK^w_1$ is $\cG^w_1$-measurable and then $\cG^b_1$-measurable, we deduce that 
$G^b_1=\{\{u, v\}: u\ne \sigma^b(1), v\in \cK^w_1\}$ is also $\cG^b_1$-measurable. This proves {\em (i)} for $k=1$. For the induction step from $k-1$ to $k$, let us first argue that $\cO^b_{k-1}\in \cG^b_k$. Indeed, $\cO^b_{k-1}$ is determined by the weights from the edge set
\[
J_{k-1}:=\big\{e=\{u, v\}: u\in \cK^b_{k-1}, v\in \cN\big(\sigma^b(k-1)\big)\big\}.
\]
Let us argue that $J_{k-1}$ is a subset of the union of $\cup_{i\le k-1}H_i$ and $F^c_{\tau^b_{k}-1}$. Combined with both facts that $\cK^b_{k-1}\in \cG^b_k$ by the induction assumption and $\cN(\sigma^b(k-1))\in \cG^w_k$ by Lemma~\ref{lem: Gk}, this will imply that $\{U_e: e\in J_{k-1}\}$ is $\cG^b_k$-measurable, and therefore $\cO^b_{k-1}\in \cG^b_k$. Indeed, the definition of $\cK^w_k$ implies that 
$V^w_n\setminus \cK^w_{k-1}\subseteq \cup_{i\le k-2}\cO^w_i\cup \Sigma^w(\tau^b_{k-1})$. It follows that 
\[
\cN\big(\sigma^b(k-1)\big) \subseteq \big(\cN\big(\sigma^b(k-1)\big)\cap \cK^w_{k-1}\big)\cup (V^w_n\setminus \cK^w_{k-1}\big) \subseteq 
\cup_{i\le k-1}\cO^w_{i} \cup \Sigma^w(\tau^b_{k-1}).
\]
We deduce that
\begin{align*}
J_{k-1}&\subseteq \bigcup_{i\le k-1}\big\{\{u, v\}: u\in \cK^b_{k-1}, v\in \cO^w_i\big\}\bigcup \big\{\{u, v\}: u\in \cK^b_{k-1}, v\in \Sigma^w(\tau^b_{k-1})\big\}\\
&\subseteq \bigcup_{i\le k-1}H_i\, \bigcup \,F^c_{\tau^b_k-1},
\end{align*}
where we have used $\cK^b_{k-1}\subseteq V^b_n\setminus \Sigma^b(\tau^b_{k-1})$ and $\tau^b_{k-1}\le \tau^b_k-1$. 
This shows that $\cO^b_{k-1}$ is $\cG^b_k$-measurable. 
Since $\Sigma(\tau^b_k)$ is $\cG^w_k$-measurable and then $\cG^b_k$-measurable, it follows that $\cK^b_k=\cK^b_{k-1}\setminus (\cO^b_{k-1}\cup \Sigma^b(\tau^b_k))$ is $\cG^b_k$-measurable. As both $\Sigma^b(\tau^b_k)$ and $\cK^w_k$ are $\cG^b_k$-measurable, so is $G^b_k$. This proves {\em (i)}. 

For {\em (ii)}, let us denote
\[
H'_i=H_i\cap F_{\tau^b_k-1}=\big\{e=\{u, v\}: u\in V^b_n\setminus \Sigma^b(\tau^b_{k}-1), v\in \cO^w_i\setminus\Sigma^w(\tau^b_k-1)\big\},
\]
and note that $\cG^b_k$ is also the sigma-algebra generated by $\cG^w_k, \cO^w_k$ and $(\{U_e: e\in H'_i\})_{i\le k-1}$, since $\{U_e: e\in F^c_{\tau^b_k-1}\}\in \cG^w_k$. 
We point out that $\cO^w_k$ is a function of the edge weights from the set 
\[
I_k=\big\{e=\{\sigma^b(k), v\}: v\in \cK^w_k\big\}.
\]
Moreover, since $\sigma^b(k)$ and $\cK^w_k$ are both $\cG^w_k$-measurable according to Lemma~\ref{lem: Gk}, we see that $I_k$ is itself $\cG^w_k$-measurable. 
Similarly, we have $H'_i\in \cG^w_k$ for all $i\le k-1$. 
From the fact that $\cO^w_i\cap \cK^w_k=\varnothing$ for all $i\le k-1$ and $\sigma^b(k)\in \Sigma^b(\tau^b_k)$, we deduce that $G^b_k$ is disjoint from both $\cup_{i\le k-1}H'_{i}$ and $I_k$. We also note that $G^b_k, I_k$ and $H'_i, i\le k-1$ are all subsets of $F_{\tau^b_k-1}$. Indeed, the case for $G^b_k$ stems from the facts that $\Sigma^b(\tau^b_k - 1)\subseteq \Sigma^b(\tau^b_k)$ and $\cK^w_k\subseteq V^w_n\setminus \Sigma^w(\tau^b_k)$. The arguments for $I_k$ and $H'_i$ are easier. We then apply Lemma~\ref{lem: Fk} {\em (iii)} to find that conditional on $\cG^w_k=\cF_{\tau^b_k-1}$, $\{U_e: e\in G^b_k\}$ are i.i.d.~uniform and independent of both $\{U_e: e\in I_k\}$ and $\{U_e: e\in H'_i\}, i\le k-1$. It follows that conditional on $\cG^w_k, \cO^w_k$ and $\{U_e: e\in H'_i\}, i\le k-1$, we still have $\{U_e: e\in G^b_k\}$ as a collection of i.i.d.~uniform variables. The conclusion follows. 
\end{proof}

We recall that  a {\em lead} vertex is a vertex with the lowest rank in its connected component. For $\ell\in \N$, let $\zeta_{\ell}$ be the rank of the $\ell$th lead vertices in $(\sigma(i))_{i\in [n]}$, with the understanding that $\zeta_{\ell'}=\infty$ for all $\ell'>\ell$ if there are only $\ell$ connected components. Thanks to Proposition~\ref{prop: prim}, we know that $\sigma(j+1)$ is a lead vertex if and only if the Prim edge $e_{j}$ identified at step $j+1$ in Algorithm~\ref{algo: Prim} satisfies $U_{e_{j}}>p$.  
It follows that $\zeta_{\ell}-1$ is a stopping time, since 
\[
\{\zeta_{\ell} \le j+1\} = \big\{ \ell \le \big|\lbrace i\in [j]: U_{e_i}>p \rbrace\big|\big\},
\]
and is therefore $\cF_j$-measurable by Lemma~\ref{lem: Fk} {\it (ii)}. We then let $\cH_{\ell}:=\cF_{\zeta_{\ell}-1}$ be the sigma-algebra generated by the events $A\cap\{\zeta_{\ell}-1\le j\}$ for all $A\in \cF_j, j\ge 0$. 

\begin{Lem}
\label{lem: Hk}
For $\ell\in \N$, the following statements hold true: 
\begin{enumerate}[(i)]
\item 
both $(\sigma(i))_{i\le \zeta_{\ell}}$ and $\zeta_{\ell}$ are $\cH_{\ell}$-measurable;
\item 
conditional on $\cH_{\ell}$ and $\zeta_{\ell}<\infty$, $\{U_e: e\in F_{\zeta_{\ell}-1}\}$ are i.i.d.~uniformly distributed on $(0, 1)$. 
\end{enumerate}
\end{Lem}

\begin{proof}
The proof of {\em (i)} is similar to that of Lemma~\ref{lem: Gk} {\em (i)}, and stems from the fact that $\zeta_{\ell}-1$ is a stopping time. We omit the detail. For {\em (ii)}, it suffices to note that on the event $\{\zeta_\ell = j+1\}$, the claim follows from Lemma~\ref{lem: Fk}. Summing over $j$ then concludes the proof. 
\end{proof}

{\small
\setlength{\bibsep}{.2em}
\bibliographystyle{plain}
\bibliography{these}
}

\end{document}